\documentclass[10pt]{article}
\usepackage{amsmath}
\usepackage{amsthm}
\usepackage{amsfonts}
\usepackage{amssymb}
\usepackage{latexsym} 
\usepackage{xcolor}
\usepackage[colorlinks, allcolors=blue]{hyperref}
\usepackage[matrix,tips,graph,curve]{xy}



\makeatletter 
\@addtoreset{equation}{section}
\makeatother

\theoremstyle{plain}
\newtheorem{theorem}[equation]{Theorem}
\newtheorem{corollary}[equation]{Corollary}
\newtheorem{lemma}[equation]{Lemma}
\newtheorem{proposition}[equation]{Proposition}

\theoremstyle{definition}

\newtheorem{remark}[equation]{Remark}

\newcommand{\IC}{\mathbb{C}}

\newcommand{\IN}{\mathbb{N}}

\newcommand{\IP}{\mathbb{P}}

\newcommand{\IR}{\mathbb{R}}

\newcommand{\IZ}{\mathbb{Z}}

\DeclareMathOperator{\spec}{spec}

\DeclareMathOperator{\tr}{tr}
\DeclareMathOperator{\Tr}{Tr}

\DeclareMathOperator{\rank}{rank}

\newcommand{\Ker}{\mathrm{Ker \,}}

\renewcommand\Re{\mathrm Re}
\renewcommand\Im{\mathrm Im}

\def\d/{/\mspace{-6.0mu}/}

\newcommand{\p}{\partial}

\title{The Asymptotic Structure of Monopoles in $\IR^3$, Calorons, 
and  Instantons on ALF Spaces }
\author{\small{Sergey A. Cherkis}\footnote{Partially supported by Simons Foundation Grant MP-TSM-00002295} \\ \scriptsize{University of Arizona}\\ \footnotesize{\textsf{cherkis@arizona.edu}} \and \small{Mark Stern}\footnote{Partially supported by Simons Foundation Grant 3553857} \\ \scriptsize{Duke University} \\ \footnotesize{\textsf{stern@math.duke.edu}}}
\date{}
\begin{document}

 \maketitle
\begin{abstract}
Abstract: We study the asymptotic structure of instantons on multi-centered Taub-NUT manifolds, calorons, and monopoles on $\IR^3$. We show that, without any assumptions on symmetry breaking, these instantons and monopoles asymptotically decompose as a sum of U(1) monopoles. 
\end{abstract}
\section{Introduction}\label{Intro}
Let $(E,A)$ be a hermitian vector bundle with a hermitian connection over an oriented Riemannian 4-manifold, $M$. We call $(E,A)$ an {\em instanton} if the curvature $F_A$ of the connection is $L^2$ and anti-self-dual. The manifold $M$ is called an {\em ALF gravitational instanton} if it is a complete hyperk\"ahler 4-manifold with cubic volume growth and  faster than quadratic  curvature decay. Instantons on such manifolds have been extensively studied. The Nahm transform in its many incarnations has been one of the main tools developed to study the moduli spaces of these instantons. 

The analysis of the Nahm transform requires a fine understanding of the asymptotic structure of the instanton connection and curvature. In particular, one generally requires that the curvature decays quadratically and that at large radius the connection decomposes as a direct sum of $U(1)$ instanton bundles up to a negligible error.  Jaffe and Taubes \cite[Theorem 10.3 and Theorem 10.5]{JT80} proved these asymptotic properties for $SU(2)$ monopoles in $\IR^3$. Until \cite{First}, these prescribed asymptotics were simply assumed for all higher rank gauge groups, leaving unaddressed whether the Nahm transform gave information about the full moduli space or only the special subset satisfying the requisite asymptotics. 

In \cite[Theorem A and Theorem B]{First}, we proved that the instanton curvature always decays quadratically and the connection has the desired asymptotics, under the assumption of {\em maximal symmetry breaking}. The ALF spaces are asymptotically circle bundles over a finite quotient of $\IR^3$. Maximal symmetry breaking is the condition that, at sufficiently large radius, the holonomy of $E$ in the $S^1$ fiber decomposes $E$ into a direct sum of eigenline bundles, with all eigenvalues having multiplicity one.  Upon dimensional reduction from instantons on $S^1\times \IR^3$ to monopoles on $\IR^3$,  maximal symmetry breaking becomes the condition that $E$ decomposes into a direct sum of eigenline bundles of the Higgs field,  with all eigenvalues having multiplicity one.  This assumption is automatically satisfied for all nontrivial $SU(2)$   monopoles on $\IR^3$. 
Most analysis of instanton moduli spaces has been carried out under the assumption of maximal symmetry breaking.
In the absence of the assumption of maximal symmetry breaking, our understanding of the Nahm transform and the moduli space of instantons is much weaker.
Stratification of monopole moduli spaces in this general case (identifying it with the stratification of rational maps) was given in \cite{Mur89}. Existence of hyperk\"ahler metrics on each stratum is conjectured in \cite{MS03}. 
Two recent papers \cite{CN22} and \cite{Me24} further advance our understanding of these moduli spaces. 
In \cite{CN22}, Charbonneau and Nagy develop the Nahm transform to construct monopoles without maximal symmetry breaking. In \cite{Me24}, Mendizabal uses Fredholm theory and inverse function theorem techniques to construct monopole moduli spaces, without assuming maximal symmetry breaking. The latter analysis uses essentially the assumption that the monopoles asymptotically decompose as 
a direct sum of $U(1)$ monopoles. 

  Our goal in this paper is to show that --- {\em without the assumption of maximal symmetry breaking } --- monopoles on $\IR^3$, calorons, and instanton connections on multi-center Taub-NUT manifolds always asymptotically decompose as a sum of $U(1)$ instantons (or monopoles), and so remove a major obstacle to the careful study of the moduli space of instantons with arbitrary asymptotic holonomy. 

 In order to state our main results, we first recall the structure of the $k$-centered Taub-NUT manifold,  $TN_k$. We recall that $TN_k$ is hyperk\"ahler with an isometric circle action with $k$ fixed points, $\{\nu_1,\ldots,\nu_k\}$. The quotient of  $TN_k$ by this circle action is $\IR^3$. Let $\Pi:TN_k\to \IR^3$, denote the quotient map. The metric $g$ can be be expressed in terms of the Euclidean metric $dx^2$ on $\IR^3$ as 
\begin{align}\label{eq:metric}
Vdx^2+V^{-1}(d\theta +\omega)^2,
\end{align}
where for some $l>0$, 
\begin{align}\label{ldef}V = l+\sum_{\sigma = 1}^k\frac{1}{2|x-\nu_\sigma|},  
\end{align}
 $\theta$ is a local coordinate on the circle fiber, and $\omega$ is a local connection one-form on the circle bundle satisfying
\begin{align}d\omega =  \ast_{\IR^3}dV.
\end{align} 
When $k=0$, $TN_k$ specializes to $\IR^3\times S^1$, and all of our results extend to this context by setting $k=0$.

We first prove the following asymptotic structure theorem for monopoles on $\IR^3$. 
\begin{theorem}\label{thmA}
Let $K\subset \IR^3$ be a compact set (possibly empty). 
Let $(E,d_A,\Phi)$ be a hermitian vector bundle over $\IR^3\setminus K$, with connection $d_A$ and Higgs field $\Phi$, satisfying the monopole equation 
\(d_A \Phi = * F_A\), with curvature $F_A\in L^2(\IR^3\setminus K)$.  Then $F_A$ decays quadratically. At large radius, $E$ decomposes into a direct sum of subbundles $E=\oplus E(\lambda,\beta)$. $\Phi$ acts on $E(\lambda,\beta)$ as  $ i\lambda +\frac{i\beta}{r} +o(\frac{1}{r})$, with each $\beta\in \frac{1}{2}\IZ$. 
Let $P_{\lambda,\beta}$ denote unitary projection onto $E(\lambda,\beta)$. Then 
for $\lambda_1\not = \lambda_2$, 
$$|P_{\lambda_1,\beta_1}d_A P_{\lambda_2,\beta_2}|= O(r^{-N}), \forall N,$$ 
and for $\beta_1\not = \beta_2$ 
$$|P_{\lambda,\beta_1}d_A P_{\lambda,\beta_2}|= o(r^{-1}).$$ 
\end{theorem}

\begin{remark}
We have allowed the exclusion of the compact set $K$ in the above theorem in order to allow its application to monopoles with Dirac singularities. 
\end{remark} 

For instantons on $TN_k$ we obtain the following asymptotic structure.  
\begin{theorem}\label{thmB}
Let $(E,d_A)$ be a hermitian vector bundle with instanton connection over $TN_k$, $k\geq 0$. Then the curvature $F_A$ decays quadratically. 
There exist a compact set $K\subset\IR^3$ and hermitian vector bundles with connection 
$(W(\lambda,\beta),d_{W(\lambda,\beta)})$ 
on $\IR^3\setminus K$ such that on $\Pi^{-1}(\IR^3\setminus K)$,
$$E={\mathop{\oplus}_{(\lambda,\beta)} } \Pi^* W(\lambda ,\beta),$$
and 
$$d_A = {\mathop{\oplus}_{(\lambda,\beta)} } \left( \Pi^*d_{W(\lambda,\beta)}+ i\lambda+\frac{i\beta}{r}+o(\frac{1}{r})\right) + O(r^{-N}), \  \forall N,$$  
with each $\beta$ satisfying $ l\beta-\frac{\lambda k}{2}\in \frac{1}{2}\IZ$. 
Let $P_{\lambda,\beta}$ denote unitary projection onto $\Pi^* W(\lambda ,\beta)$. Then 
for $\lambda_1\not = \lambda_2$, 
$$ |P_{\lambda_1,\beta_1}d_A P_{\lambda_2,\beta_2}|= O(r^{-N}), \forall N,$$ 
and for $\beta_1\not = \beta_2$ 
$$|P_{\lambda,\beta_1}d_A P_{\lambda,\beta_2}|= o(r^{-1}).$$ 
\end{theorem}
 \begin{remark}
Note that, due to a minor change in convention,  $\lambda$ in the above theorem corresponds to $-\lambda$ in \cite[Theorem B]{First}.
\end{remark}
 \begin{remark}
In Theorems \ref{thmA} and \ref{thmB}, the error term is given as $o(\frac{1}{r})$. In the gauge theory literature, the stronger assumption of an $O(r^{-2})$ error term is often assumed and such a stronger estimate was proved in \cite[Theorem B]{First}, under the assumption of maximal symmetry breaking. It is unclear whether the weaker error bound in the current case  is an artifact of the techniques employed or reflects new phenomena in the case of nonmaximal symmetry breaking. 
\end{remark}
Theorem \ref{thmA} is, of course,  a special case of Theorem \ref{thmB}. Nonetheless, we first give a separate proof of Theorem \ref{thmA} and then prove Theorem \ref{thmB}, as we think the second proof is easier to understand once the reader has seen the first.

The proof of Theorem \ref{thmA} requires a study of the interplay between several quantities.  For a ball \(B_R(0)\) of radius \(R\) and its complement \(B_R(0)^c\), set
$$T(R):= \int_{B_R(0)^c}|F_A|^2dv,$$
and
$$S(R):= \|F_A\|_{L^\infty(B_R(0)^c)}.$$
Moser iteration tells us that for   $n\in \IN$,  we have for some $C>0$, 
$$S(R)\leq C(\frac{4n^2}{R^2}+S(\frac{R}{2})^{(\frac{3}{4})^n}T(\frac{R}{2})^{2-2(\frac{3}{4})^n}).$$ To control $T(R)$, it is natural to consider the divergence theorem applied to the relation 
$$\Delta \frac{1}{2}|\Phi|^2= -|F_A|^2,$$
to express $T(R)$ in terms of  boundary integrals of $\langle \ast F_A(\frac{\p}{\p r}),\Phi\rangle$ .  Unfortunately, the fact that for nontrivial monopoles $\lim_{|x|\to\infty}|\Phi|(x)\not = 0$ makes it hard to work with these boundary integrals. So instead, we first show that there is a well defined limiting spectrum $\spec (\Phi(\infty))$ for $\Phi(x)$ as $|x|\to\infty$ and then, for $i\lambda\in \spec (\Phi(\infty))$, define spectrally truncated operators $\Phi_\lambda:= P_\lambda(\Phi -i\lambda I)P_\lambda,$ where $P_\lambda$ is a spectral projection onto a sufficiently small neighborhood $i\lambda$.  Computing $\Delta|\Phi_\lambda|^2$ in terms of the curvature and applying the divergence theorem to this expression then provides enough feedback between   $\sum_{i\lambda\in \spec (\Phi(\infty))}\|\Phi_\lambda\|_{L^\infty (B_R(0)^c)}, S(R)$, and $T(R)$ for us to be able to obtain the desired quadratic curvature decay and $O(\frac{1}{r})$ bounds for the eigenvalues of  $\Phi_\lambda$. To pass from 
 $\spec (\Phi(x))=\{i \lambda+O(\frac{1}{r}):i\lambda\in \spec (\Phi(\infty))\}$ to $\spec (\Phi(x))=\{i \lambda+i \frac{m}{2 r}+o(\frac{1}{r}):i\lambda\in \spec (\Phi(\infty)),  m \in \mathfrak{m} \}$,     
  for some $\mathfrak{m}\subset \IZ$, requires a further refined spectral truncation of $\Phi$ and an application of Chern-Weil theory to quantize the $m$ by relating them to the first Chern numbers of subbundles defined by the spectral truncations. 

In the analysis of the refined spectral truncation, some complications arise from the fact that norms of derivatives of spectral projections  
grow schematically like $|F_A||\frac{1}{\lambda_a}-\frac{1}{\lambda_b}|$ where $i\lambda_a$ is an eigenvalue of $\Phi$ whose eigenvector lies in the image of the projection and  $i\lambda_b$ is an eigenvalue of $\Phi$ whose eigenvector lies outside this image. When $i\lambda$ has multiplicity greater than $1$, we can have $|\frac{1}{\lambda_a}-\frac{1}{\lambda_b}|=O(r)$. (Here the eigenvalues lie in $\spec (\Phi(x))$ for some finite $x$ - not in $\spec (\Phi(\infty))$.) This weakens many of our estimates and is the primary reason our error terms in both  theorems are of the form $o(\frac{1}{r})$ rather than the $O(\frac{1}{r^2})$ we obtained in \cite[Theorem B]{First}.

In order to extend the proof of Theorem \ref{thmA} to instantons on $TN_k$, we need a replacement for the Higgs field $\Phi$. The dimensional reduction origins of $\Phi$ suggest that the replacement should be $\nabla_{\theta}:= \nabla_{\frac{\p}{\p\theta}}.$
  Of course this is an unbounded operator and therefore an unsuitable replacement. On the other hand, the analogs of the spectrally truncated operators $\Phi_\lambda$ are  Hilbert-Schmidt  operators acting on a vector bundle over $\IR^3\setminus K$, $K$ compact, whose fiber at $b$ is a suitable finite rank subspace of the Hilbert space $L^2(\Pi^{-1}(b),E)$ defined by spectral data. To define such spectral truncations in the complement of a compact set requires first proving that $\spec ((\nabla_\theta)_{L^2(\Pi^{-1}(b),E)})$ converges to a well defined limiting discrete spectrum as $|b|\to \infty$. In Subsection \ref{subsecsf}, we prove such a limit exists by studying the behavior of spectral $\zeta$ functions. Given the existence of the spectral limit, we define $\Phi_\lambda^I:= P_\lambda^I(\nabla_\theta -i\lambda I)P_\lambda^I$ as suggested by the monopole case, for a suitable spectral projection $P_\lambda^I$. We then show that $\Phi_\lambda^I$ satisfies equations so similar to its monopole forebear that the proof of Theorem \ref{thmB} is essentially the same as that of Theorem \ref{thmA}, and we only repeat those details where there is some nontrivial deviation from the monopole case.

\section{Monopoles in \texorpdfstring{$\IR^3$}{R³}}
Let $\mathcal{E}$ be a hermitian vector bundle on 
$ \IR^3\setminus B_{R_0}(0)$.     
Let $(d_A,\Phi)$ be a smooth hermitian connection and skew symmetric endomorphism  on $E$ satisfying the monopole equation
\begin{align}\label{mono1}
d_A\Phi = \ast F_A.
\end{align}
\begin{remark}
The most natural assumption for analyzing monopoles is that  $F_A\in L^2$ (or $L^2$ outside a compact set if we wish to allow for Dirac singularities, for example). In our initial treatment of monopoles, we will make the additional simplifying assumption that the Higgs field is bounded. Subsequently, when we consider instantons on $TN_k$, and on $\IR^3\times S^1$, there is no Higgs field available for which we can assume a bound. The quadratic curvature bound that we prove for instantons on $TN_k$, $k\geq 0$, immediately implies, upon lifting a monopole to a caloron, that the original Higgs field was bounded. Hence, since Theorem \ref{thmA} is a special case of  Theorem \ref{thmB}, we see that the theorem holds without the hypothesis that $|\Phi|$ is bounded. 
\end{remark}

Assume $F_A\in L^2(\IR^3\setminus B_{R_0}(0)).$ For $R>R_0$, define the monotone decreasing quantities
$$T(R):= \int_{B_R(0)^c}|F_A|^2dv,$$
and 
$$S(R):= \|F_A\|_{L^\infty(B_R(0)^c)}.$$
The monopole equation is a dimensional reduction of the instanton equation. Hence any monopole $(E,d_A,\Phi)$ lifts to an instanton $(\tilde E,d_A+d\theta(\frac{\p}{\p\theta} + \Phi))$ on $S^1\times \IR^3$. Hence  Uhlenbeck's $\epsilon-$regularity theorem for Yang-Mills connections extends to our context.   
\begin{lemma}[See for example {{\cite[Theorem 2.2.1]{Tian}}}]\label{epreg} 
Let $(d_A,\Phi)$ be a smooth monopole with $L^2$ curvature on a hermitian  vector bundle $\mathcal{E}$ over $\IR^3\setminus B_{R_0}(0)$. There exist constants $\epsilon, C>0$  such that for every $\rho>0$ and $ p\in \IR^3\setminus B_{R_0+\rho}(0)$, 
for which
\begin{align}\label{epsreg}\int_{B_\rho(p)}|F_A|^2dv<\frac{\epsilon}{\rho+1},
\end{align}
we have
\begin{align}\label{epsregc} |F_A|^2(p)\leq\frac{C}{\rho^3}\int_{B_\rho(p)}|F_A|^2dv.
\end{align}
\end{lemma}  
\begin{remark}
The replacement of the usual condition  $\int_{B_\rho(p)}|F_A|^2dv< \epsilon$ by\\ $\int_{B_\rho(p)}|F_A|^2dv<\frac{\epsilon}{\rho+1}$  arises from lifting the monopole  to a $\IZ$-periodic instanton on $\IR^4$ (to avoid the injectivity constraints in \cite[Theorem 2.2.1]{Tian}). The integral of $|F_A|^2$ on the lifted ball is bounded by $(1+\rho)$ times the integral of $|F_A|^2$ on $B_\rho\subset \IR^3$. Alternatively if one runs the $\epsilon$-regularity argument on $\IR^3$, there is a change of scaling in Moser iteration due to the dimension drop that leads to the altered constraint. 
\end{remark}
\begin{corollary}\label{coro1}
$\lim_{R\to\infty}S(R) = 0.$
\end{corollary}
\begin{proof}
 Choose $R_1>R_0$ so that $T(R_1)<\frac{\epsilon}{2},$ $\epsilon$ as given in Lemma \ref{epreg}. Then applying \eqref{epsregc} with $\rho = 1,$ 
we see that for $p\in B_{R_1+1}(0)^c$, $|F_A|^2(p)\leq CT(|p|-1).$ Since $T(R)\to 0$ as $R\to \infty$ the result follows. 
\end{proof}

We will have frequent need to refer to Moser iteration in our estimates. For the convenience of the reader, we reproduce (with minor notational modifications and corrections) from 
\cite[Proposition 2]{First} a convenient formulation of this elliptic estimate. 
\begin{proposition}[Moser iteration]\label{whorse}
Let $(M,g)$ be a smooth complete Riemannian manifold of dimension \(n\) with bounded geometry. Let $\delta(M)$ denote the injectivity radius of $M$. Let $ S_M>0$ satisfy  for all $ p\in M$, $ R<\frac{\delta(M)}{2} $, and all  $ \xi\in C_c^\infty(B_{R}(p))$,  
\begin{align}\label{sob3}S_M \|d \xi\|^2_{L^2(B_R(p))}\geq \| \xi\|_{L^{\frac{2n}{n-2}}(B_R(p))}^2.
\end{align}
Let  $ R_1<R_2<\frac{\delta(M)}{2} $.
 Suppose $f$ is a nonnegative function satisfying 
\begin{align}\label{seed}\Delta f\leq w^2f,\end{align}
 for  a nonnegative function  $w\in L^{\infty}(B_{R_2}(p))$. Set 
 $W:=(R_2-R_1)\|   w\|_{L^{\infty}(B_{R_2}(p))}.$  
Then 
\begin{align}\label{moserdoneii}
\|f\|_{L^{\infty}(B_{R_1}(p))}\leq  
 (R_2-R_1)^{-n} S_M^{\frac{n}{2}}( 1 + W^2)^{\frac{n}{2}}\frac{ 2^{\frac{n^2+2}{2}}  }{(1-(\frac{1}{4})^{\frac{1}{n}})^{n}}  \| f \|_{L^1(B_{R_2}(p))}.
\end{align}
\end{proposition}
Specializing to $f= |F_A|^2$ gives  the primary estimate behind the proof of Lemma \ref{epreg}.
\begin{lemma}\label{moseri}
Let $B_L(p)\subset B_{R_0}(0)^c$. There exists $C_2>0$ such that  $\forall\delta\in (0,1],$ we have 
\begin{align}
\|F_A\|^2_{L^\infty(B_{(1-\delta)L}(p))}\leq \frac{C_2}{\delta^3}(\frac{1}{L^2}+\delta^2\|F_A\|_{L^\infty(B_{L}(p))})^{\frac{3}{2}}\|F_A\|^2_{L^2(B_{L}(p))}.
\end{align}
\end{lemma}
\begin{corollary}\label{goodcoro}For $r_2>r_1>R_0$, 
\begin{align}S(r_2)\leq  C_2 (\frac{1}{(r_2-r_1)^2}+ S(r_1))^{\frac{3}{4}}\sqrt{T(r_1)}.\end{align}
In particular, if $\frac{1}{(r_2-r_1)^2}\leq S(r_1),$ then we have 
\begin{align}\label{oddineq}S(r_2)\leq 2C_2S(r_1)^{\frac{3}{4}} \sqrt{T(r_1}).\end{align}
Iterating \eqref{oddineq} gives for a sequence $\{r_j\}_{j=1}^{n+1}$ with $r_j<r_{j+1}$ and $\frac{1}{(r_{j+1}-r_j)^2}\leq S(r_j),$ 
\begin{align}\label{oddineq2}S(r_{n+1})\leq   C_3S(r_1)^{(\frac{3}{4})^{n}} T(r_1)^{2-2(\frac{3}{4})^{n}} .\end{align}
If $S(R)\geq \frac{4n^2}{R^2}$, the sequence $r_j= \frac{R}{2}+\frac{(j-1)R}{2n}$ satisfies the hypotheses of \eqref{oddineq2}, and we have 
\begin{align}\label{oddineq3}S(R)\leq   C_3S(\frac{R}{2})^{(\frac{3}{4})^{n}} T(\frac{R}{2})^{2-2(\frac{3}{4})^{n}} .\end{align}
\end{corollary}
The proof of Lemma \ref{moseri} is standard, and we will not repeat it here. (See for example \cite[Proposition 2]{First}.) 

The Bianchi identity and \eqref{mono1} imply the Bochner formula (which is used infinitely many times in the proof of Lemma \ref{moseri}) :
\begin{align}\label{boch-1}
 \Delta\frac{1}{2}|F_A|^2 = -|\nabla F_A|^2-|[\Phi,F_A]|^2+3\langle[F_{jk},F_{km}],F_{jm}\rangle .
\end{align}
Let $\{w_{a}\}_a$ be a local unitary $\Phi$-eigenframe for $E$, with $\Phi w_a = i\lambda_a w_a$. Expand $F_A$ in this frame as 
$F_A = F_a^bw^b\otimes w_a^\dagger.$ With this notation, \eqref{boch-1} becomes
\begin{align}\label{boch-0}
 \Delta\frac{1}{2}|F_A|^2 = -|\nabla F_A|^2-(\lambda_a-\lambda_b)^2|  F_a^b|^2+3\langle[F_{jk},F_{km}],F_{jm}\rangle .
\end{align}
 Observe that we can use \eqref{boch-1} to obtain information about  the relative size of different components of $F$. Let $\eta $ be a smooth compactly supported function\label{page:eta} with $\eta(t) =1$ for $|t|\leq \frac{1}{2}$, $\eta(t) =0$ for $|t|\geq 1$, and $|d\eta|\leq 4$. For $R_2>R_1>R_0$, set $\eta_{R_1,R_2}(r):= \eta(\frac{1}{2}+\max\{0,\frac{r-R_1}{2(R_2-R_1)}\})$. Then integrating \eqref{boch-0} against $(1-\eta_{R_1,R_2})^2$, for $R>2R_0$ yields 
\begin{align}\label{offdiag1}
&\int_{B_{R_1}(0)^c}(|\nabla ((1-\eta_{R_1,R_2})F_A)|^2 +(1-\eta_{R_1,R_2})^2(\lambda_a-\lambda_b)^2| F_a^b|^2)dv\nonumber\\
&= \int_{B_{R_1}(0)^c}(|d\eta_{R_1,R_2}|^2|F_A|^2+3(1-\eta_{R_1,R_2})^2\langle[F_{jk},F_{km}],F_{jm}\rangle)dv\nonumber\\
&\leq  64(\frac{1}{(R_2-R_1)^2}+S(R_1))T(R_1).
\end{align}  
This suggests that the off block diagonal components of $F_A$ (in a $\Phi$-eigenframe) are smaller than the block diagonal components. When all the eigenvalues of $\Phi$ are distinct, the block diagonal components of $F_A$ become simply diagonal, and therefore the cubic term $\langle[F_{jk},F_{km}],F_{jm}\rangle$ should decay more rapidly than $O(|F_A|^3)$.  This is the primary reason  why  the maximal symmetry breaking case is easier to analyze.

\section{Analyzing \texorpdfstring{$\Phi$}{Phi}}
We use spectral projections associated with $\Phi$ in later estimates. In order to control these, we next show that the spectrum of $\Phi(p)$ converges as $|p|\to \infty$. 
For $y\in\IR^3,$ set 
$$r_y(x):= |x-y|, \text{ and }r(x) := |x|.$$

\begin{proposition}\label{togap}
Let $p,q\in B_R(0)^c$, $R>4R_0$. Then for all $m,n\in \IN$, $\exists C_{m,n,A}>0$ such that 
\begin{align}\label{pinch}
|\tr \Phi^m(p) -\tr \Phi^m(q)|\leq C_{m,n,A}(\frac{1}{R}+T(\frac{R}{4})^{2-2(\frac{3}{4})^{n-1}}).
\end{align}
\end{proposition}
\begin{proof}
Let $\rho>|p|,|q|>2L>2R_0$. Compute 
\begin{align}
&\int_{L<|x|<\rho}(\frac{1}{r_p}-\frac{1}{r_q})(-\Delta )\tr \Phi^m dv= 4\pi\tr \Phi^m(q)-4\pi\tr \Phi^m(p)\nonumber\\
& + \int_{S_\rho(0)}\left((\frac{1}{r_p}-\frac{1}{r_q})\frac{\p}{\p r}\tr \Phi^m +\tr \Phi^m (\frac{\rho^2-\langle p,x\rangle}{r_p^3\rho}  -\frac{\rho^2-\langle q,x\rangle}{r_q^3\rho}) \right)d\sigma \nonumber\\
&-\int_{S_L(0)}\left((\frac{1}{r_p}-\frac{1}{r_q})\frac{\p}{\p r}\tr \Phi^m +\tr \Phi^m (\frac{L^2-\langle p,x\rangle}{r_p^3L} 
 -\frac{L^2-\langle q,x\rangle}{r_q^3L})\right)d\sigma.
\end{align}
Take the limit as $\rho\to \infty$ to get 
\begin{align}
&4\pi\tr \Phi^m(p)-4\pi\tr \Phi^m(q)) = \int_{B_L(0)^c}(\frac{1}{r_p}-\frac{1}{r_q})\Delta \tr \Phi^m dv \nonumber\\
&-\int_{S_L(0)}\left((\frac{1}{r_p}-\frac{1}{r_q})\frac{\p}{\p r}\tr \Phi^m +\tr \Phi^m (\frac{L^2-\langle p,x\rangle}{r_p^3L}  -\frac{L^2-\langle q,x\rangle}{r_q^3L})\right)d\sigma.
\end{align}
The $S_L(0)$ integral is $O(\frac{L^2 }{R});$ so, we are left to estimate the  $B_L(0)^c$ integral. We use the crude estimate $|\Delta \tr \Phi^m|\leq c_{m,A}|F_A|^2,$ for some $c_{m,A}>0$. It suffices to  estimate the integral of 
$\frac{1}{r_p}|\Delta \tr \Phi^m|$, as the estimate of the other term is identical.  
Writing $B_L(0)^c= B_{\frac{|p|}{2}}(p)\cup \left(B_L(0)^c\setminus B_{\frac{|p|}{2}}(p)\right),$ we have 
$$\int_{B_L(0)^c}|\frac{1}{r_p} \Delta \tr \Phi^m| dv  
\leq c_{m,A}\int_{B_{\frac{|p|}{2}}(p)}\frac{1}{r_p}|F_A|^2dv+O(\frac{1}{R}).$$
If  $T(\frac{|p|}{4})>\frac{2}{|p|}$ and $S(\frac{|p|}{2})>\frac{16(n-1)^2}{|p|^2}$, we decompose $B_{\frac{|p|}{2}}(p)$ as a union $B_{\frac{1}{T(\frac{|p|}{4})}}(p)\cup \left(B_{\frac{|p|}{2}}(p)\setminus B_{\frac{1}{T(\frac{|p|}{4})}}(p)\right).$ We estimate 
\begin{align}
\int_{B_{\frac{1}{T(\frac{|p|}{4})}}(p)}\frac{1}{r_p}|F_A|^2dv \leq \frac{2\pi S^2(\frac{|p|}{2})}{T^2(\frac{|p|}{4})}
\leq C_4 T^{2-4(\frac{3}{4})^{n-1}}(\frac{|p|}{4}),
\end{align}
and 
\begin{align}
\int_{B_{\frac{|p|}{2}}(p)\setminus B_{\frac{1}{T(\frac{|p|}{4})}}(p)}\frac{1}{r_p}|F_A|^2dv \leq CT(\frac{|p|}{4})^2,
\end{align}
by the definition and monotonicity of $T(r)$. If $T(\frac{|p|}{4})\leq\frac{2}{|p|}$ or $S(\frac{|p|}{2})\leq\frac{16(n-1)^2}{|p|^2}$, then we can easily estimate $\int_{B_{\frac{|p|}{2}}(p)}\frac{1}{r_p}|F_A|^2dv$ by $O(n^4R^{-1})$. 
\end{proof}
\begin{corollary}\label{spec0}
The limit as $R\to\infty$ of the spectrum of $\Phi(Ru)$ is well defined and independent of $u$ on the unit sphere. 
Denote this limit $\spec (\Phi(\infty)).$ Let $i\lambda \in \spec (\Phi(\infty))$. 
Let $\delta(\lambda):=\frac{1}{4}\min_{i\nu\in \spec ( \Phi(\infty))\setminus\{ i\lambda\}}|\lambda-\nu|$. There exists $R_{00}>R_0$ so that for $|p|>R_{00}$, $-i\Phi(p)$ has no eigenvalues in 
$[\lambda-2\delta(\lambda), \lambda-\frac{\delta(\lambda)}{2} ]\cup [\lambda+\frac{\delta(\lambda)}{2}, \lambda+2\delta(\lambda)]$. 
\end{corollary}

\section{Spectral truncation}
 In this and subsequent sections, most computations and many definitions require that we only consider points in $B_{R_{00}}(0)^c$, $R_{00}$ as introduced in Corollary \ref{spec0}. To avoid endless repetitions of this and similar constraints, our convention henceforth is that quantities and definitions are considered  outside a sufficiently large ball, as needed.   
Let $i\lambda\in \spec (\Phi(\infty)).$ Let $\{i\lambda_a(x)\}_a$ denote the eigenvalues of $\Phi(x)$. Set 
$$P_{\lambda}:= \frac{1}{2\pi i}\int_{|z| = \delta(\lambda)}(z+i\lambda -\Phi)^{-1}dz,$$
and
$$\Phi_{\lambda}:= \frac{1}{2\pi i}\int_{|z| = \delta(\lambda)}z(z+i\lambda -\Phi)^{-1}dz=P_\lambda(\Phi -\lambda I)P_\lambda.$$
Set 
$$F^\lambda: = P_\lambda F P_\lambda, \text{  }F^D:= \sum_\lambda F^\lambda, \text{ and }F^B:= F_A-F^D.$$
$$T^\lambda(R):= \int_{B_R(0)^c}|F^\lambda|^2dv,$$
$$T^D(R):= \sum_{i\lambda\in \spec (\Phi(\infty))} T^\lambda(R),$$
and 
$$T^B(R):= T(R) - T^D(R).$$ 
Let $\delta_{min}:=\min\{\delta(\lambda):i\lambda\in \spec (\Phi(\infty))\}.$ Then for $R_2>R_1$, \eqref{offdiag1} yields
\begin{align}\label{offdiag1.5}
 T^B(R_2)&\leq  \delta_{min}^{-2}64(\frac{1}{(R_2-R_1)^2}+S(R_1))T(R_1).
\end{align}  
For every $n\in \IN$, combining \eqref{offdiag1}, {\eqref{oddineq3}}, and Corollary \ref{spec0} shows that for $R_2>R_1>\frac{3R_2}{4}$, %and $\frac{R_2}{2}>R_{00}$, 
there exists  $C_n>0$ such that 
\begin{align}\label{offdiag2}
 T^B(R_2)&\leq     \delta_{min}^{-2}C_n(\frac{1}{ (R_2-R_1)^2}+S(\frac{R_2}{2})^{(\frac{3}{4})^{n}}T(\frac{R_2}{2})^{2-2(\frac{3}{4})^{n}})T(R_1).
\end{align}  
We compute 
\begin{align}
d_AP_\lambda = \left[\sum_{|\lambda - \lambda_b|<\delta}\frac{1}{i\lambda_b-i\lambda_a}-\sum_{|\lambda - \lambda_a|<\delta}\frac{1}{i\lambda_b-i\lambda_a}\right](\ast F_A)_a^be_b\otimes e_a^\dagger.
\end{align}
 
\begin{align}\label{lambdamon}
d_A\Phi_\lambda &= \frac{1}{2\pi i}\int_{|z| = \delta(\lambda)}z(z+i\lambda -\Phi)^{-1}d_A\Phi (z+i\lambda -\Phi)^{-1}dz\nonumber\\
&= \ast F^\lambda +\sum_{\substack{|\lambda - \lambda_a|<\delta(\lambda) \\ |\lambda - \lambda_b|>\delta(\lambda)}} \frac{(\lambda_a-\lambda)}{\lambda_a-\lambda_b} [(\ast F_A)_a^bw_b\otimes w_a^\dagger 
+ (\ast F_A)_b^aw_a\otimes w_b^\dagger ].
 \end{align}
\begin{align}
d\frac{1}{2}|\Phi_\lambda|^2 &= \Re \langle\Phi_\lambda,(\ast F_A^\lambda)(\frac{\p}{\p x^j})\rangle dx^j.
\end{align}
\begin{align}\label{metoo}
d_A^*d_A\Phi_\lambda &= - \frac{1}{\pi i}\int_{|z| = \delta(\lambda)}z(z+i\lambda -\Phi)^{-1}\nabla_j\Phi (z+i\lambda -\Phi)^{-1}\nabla_j\Phi(z+i\lambda -\Phi)^{-1}dz\nonumber\\
&= - \frac{1}{\pi i}\int_{|z| = \delta(\lambda)}\frac{z dz  (\ast F_A)(\p_j)_a^b (\ast F_A)(\p_j)_s^a e_b\otimes e_s^\dagger}{(z+i\lambda -i\lambda_b)(z+i\lambda -i\lambda_a) (z+i\lambda -i\lambda_s)} .
 \end{align}
From \eqref{metoo} it is easy to see that for some $c>0$, 
\begin{align}\label{metoo2}
|d_A^*d_A\Phi_\lambda |&\leq  \frac{c}{\delta^2(\lambda)}|F_A^B||F_A|.
\end{align}
In order to relate $T(R)$ to surface integrals, we first compute $\Delta\frac{1}{2}|\Phi_{\lambda}|^2. $
\begin{align}\label{trunclap}
\Delta \frac{1}{2}\tr  \Phi_\lambda^2 &= \Delta  \frac{1}{4\pi i}\int_{|z| = \delta(\lambda)}\tr  z^2(z+i\lambda -\Phi)^{-1}dz\nonumber\\
&= - \frac{1}{2\pi i}\int_{|z| = \delta(\lambda)}\tr  z^2(z+i\lambda -\Phi)^{-2}(\ast F_A)(\frac{\p}{\p x^j})(z+i\lambda -\Phi)^{-1}(\ast F_A(\frac{\p}{\p x^j})dz\nonumber\\
&=  \sum_{a,b}|(\ast F_A)_a^b|^2\frac{1}{2\pi i}\int_{|z| = \delta(\lambda)} z^2(z+i\lambda -i\lambda_a)^{-2}(z+i\lambda -i\lambda_b)^{-1}dz  \nonumber\\
&=   \sum_{a,b}|(\ast F_A)_a^b|^2\frac{1}{2\pi i}\int_{|z| = \delta(\lambda)} \frac{z}{i\lambda_a-i\lambda_b}[(z+i\lambda -i\lambda_a)^{-1}-(z+i\lambda -i\lambda_b)^{-1}]dz  \nonumber\\
&=  | F_A^\lambda|^2  - 2|(P_\lambda  F_A (I-P_\lambda))_{a}^b|^2 \frac{ (\lambda_a-\lambda)}{\lambda_a-\lambda_b}  .
\end{align}
For $R<L$, integrate \eqref{trunclap} to get 
\begin{align}\label{unwtint}
&\int_{B_L(0)\setminus B_R(0)}\left[|  F_A^\lambda|^2  - 2|(P_\lambda  F_A (I-P_\lambda))_{a}^b|^2 \frac{ (\lambda_a-\lambda)}{\lambda_a-\lambda_b}  \right]dv\nonumber\\
&=\int_{S_L(0)}\langle \Phi_\lambda, (\ast F)(\frac{\p}{\p r})\rangle d\sigma - \int_{S_R(0)}\langle \Phi_\lambda, (\ast F)(\frac{\p}{\p r})\rangle d\sigma .
\end{align}
Also consider the weighted integral 
\begin{align}\label{wtint}
&\int_{B_L(0)\setminus B_R(0)}\frac{1}{r}\left[|  F_A^\lambda|^2  - 2|(P_\lambda  F_A (I-P_\lambda))_{a}^b|^2 \frac{ (\lambda_a-\lambda)}{\lambda_a-\lambda_b}  \right]dv\nonumber\\
&=\int_{S_L(0)}[L^{-1}\langle \Phi_\lambda, (\ast F)(\frac{\p}{\p r})\rangle + L^{-2}\frac{1}{2}|\Phi_\lambda|^2]d\sigma \nonumber\\
&\quad- \int_{S_R(0)}[R^{-1}\langle \Phi_\lambda, (\ast F)(\frac{\p}{\p r})\rangle -R^{-2}\frac{1}{2}|\Phi_\lambda|^2]d\sigma.
\end{align}
We can choose a sequence $L=L_n\to\infty$, so that the $S_L$ integrals limit to zero in \eqref{wtint}, yielding
\begin{align}
&\int_{B_R(0)^c}\frac{1}{r}[|  F_A^\lambda|^2  - 2|(P_\lambda  F_A (I-P_\lambda))_{a}^b|^2 \frac{ (\lambda_a-\lambda)}{\lambda_a-\lambda_b}  ]dv\nonumber\\
&=  - R^{-1}\int_{S_R(0)}\langle \Phi_\lambda, (\ast F)(\frac{\p}{\p r})\rangle d\sigma 
-R^{-2}\int_{S_R(0)}\frac{1}{2}|\Phi_\lambda|^2d\sigma \label{unsimp}\\
& = -\frac{\p}{\p R}R^{-1}\int_{S_R(0)}\frac{1}{2}|\Phi_\lambda|^2d\sigma \label{limwtint}.
\end{align}
Integrate \eqref{limwtint} from $R_1$ to $R_2$, with $R_2>R_1$ to get 
\begin{align}
&\int_{B_{R_2}(0)^c}\frac{R_2-R_1}{r} \left[|  F_A^\lambda|^2  - 2|(P_\lambda  F_A (I-P_\lambda))_{a}^b|^2 \frac{ (\lambda_a-\lambda)}{\lambda_a-\lambda_b}  \right]dv\nonumber\\
&\quad +\int_{B_{R_2}(0)\cap B_{R_1}(0)^c}\frac{r-R_1}{r} \left[|  F_A^\lambda|^2  - 2|(P_\lambda  F_A (I-P_\lambda))_{a}^b|^2 \frac{ (\lambda_a-\lambda)}{\lambda_a-\lambda_b}  \right]dv\nonumber\\
& = R_1^{-1}\int_{S_{R_1}(0)}\frac{1}{2}|\Phi_\lambda|^2d\sigma - R_2^{-1}\int_{S_{R_2}(0)}\frac{1}{2}|\Phi_\lambda|^2d\sigma\label{limwtint2}.
\end{align}
Hence
$R^{-1}\int_{S_{R}(0)}\frac{1}{2}|\Phi_\lambda|^2d\sigma$ is bounded, as $R\to\infty$. 
Therefore, we can bound the integral 
$\int_{S_L(0)}\langle \Phi_\lambda, (\ast F)(\frac{\p}{\p r})\rangle d\sigma $ by  
$(L^{-1}\int_{S_L(0)} |\Phi_\lambda|^2d\sigma)^{\frac{1}{2}} (L\int_{S_L(0)} |F_A|^2 d\sigma )^{\frac{1}{2}}.$ Since $|F_A|^2$ is integrable, there exists  $L=L_n\to\infty$, with $\int_{S_{L_n}(0)}|F_A|^2d\sigma \leq \frac{1}{L_n\ln(L_n)}.$ This implies $\int_{S_{L_n}(0)}\langle \Phi_\lambda, (\ast F)(\frac{\p}{\p r})\rangle d\sigma \to 0$, and we can take the limit of \eqref{unwtint} to get 
\begin{align}\label{unwtintl}
&\int_{ B_R(0)^c}[|  F_A^\lambda|^2  - 2|(P_\lambda  F_A (I-P_\lambda))_{a}^b|^2 \frac{ (\lambda_a-\lambda)}{\lambda_a-\lambda_b}  ]dv\nonumber\\
&= - \int_{S_R(0)}\langle \Phi_\lambda, (\ast F)(\frac{\p}{\p r})\rangle d\sigma \nonumber\\
&= - R^2\frac{\p}{\p R}\left(R^{-2}\int_{S_R(0)}\frac{1}{2}| \Phi_\lambda|^2d\sigma\right).
\end{align}
Multiply \eqref{unwtintl} by $R^{-2}$ and integrate to get 
\begin{align}\label{unwtintl2}
& \int_{ B_R(0)^c}(1-\frac{R}{r})[|  F_A^\lambda|^2  - 2|(P_\lambda  F_A (I-P_\lambda))_{a}^b|^2 \frac{ (\lambda_a-\lambda)}{\lambda_a-\lambda_b}  ]dv\nonumber\\
&= R^{-1}\int_{S_R(0)}\frac{1}{2}| \Phi_\lambda|^2d\sigma.
\end{align}
Hence  
\begin{align}\label{unwtintlconsec} 
&   R^{-1}\int_{S_R(0)}\frac{1}{2}| \Phi_\lambda|^2d\sigma\leq  T^\lambda(R)-\int_{ B_R(0)^c} \frac{R}{r} |F_A^\lambda|^2dv+2T^B(R), 
\end{align}
and 
\begin{align}\label{unwtintlsum}
&   R^{-1}\int_{S_R(0)}\frac{1}{2}\sum_\lambda| \Phi_\lambda|^2d\sigma\leq  T^D(R)-\int_{ B_R(0)^c} \frac{R}{r} |F_A^D|^2dv+2T^B(R).
\end{align}

This then implies that 
\begin{align}\label{pitstop}
 |\int_{S_R(0)}\langle \Phi_\lambda,(\ast F_A)(\frac{\p}{\p r})\rangle d\sigma| &\leq  (R^{-1}\int_{S_R(0)}|\Phi_\lambda|^2 d\sigma)^{\frac{1}{2}}(R\int_{S_R(0)} | F_A^\lambda|^2d\sigma)^{\frac{1}{2}}\nonumber\\
&\leq  \frac{1}{2}R^{-1}\int_{S_R(0)}|\Phi_\lambda|^2 d\sigma+\frac{1}{2}R\int_{S_R(0)} | F_A^\lambda|^2d\sigma\nonumber\\
&= \frac{1}{2}R^{-1}\int_{S_R(0)}|\Phi_\lambda|^2 d\sigma-\frac{1}{2}R\frac{\p}{\p R}T^\lambda(R).
\end{align}
Combine \eqref{pitstop}, \eqref{unwtintlconsec}, and \eqref{unwtintl} and sum over $\lambda$  to obtain 
\begin{align}\label{sunwtintl}
&T^D(R)  - 2T^B(R)\nonumber\\
&\leq \frac{1}{2}R^{-1}\int_{S_R(0)}\sum_\lambda| \Phi_\lambda|^2 d\sigma-\frac{1}{2}R\frac{\p}{\p R}T^D(R)\nonumber\\
&\leq \frac{1}{2}T^D(R)-\frac{1}{2}\int_{ B_R(0)^c} \frac{R}{r} |F_A^D|^2dv+ T^B(R)-\frac{1}{2}R\frac{\p}{\p R}T^D(R).
\end{align}
Hence 
\begin{align}\label{dortmund}
&\frac{\p}{\p R}(RT^D(R)) \leq  6T^B(R)- \int_{ B_R(0)^c} \frac{R}{r} |F_A^D|^2dv,
\end{align}
and for all $\epsilon $, 
\begin{align}\label{bochum}
&\frac{\p}{\p R}( R^{1-\epsilon}T^D(R)) \leq  6R^{-\epsilon}T^B(R)- R^{-\epsilon}\int_{ B_R(0)^c} (\epsilon+\frac{R}{r}) |F_A^D|^2dv.
\end{align}

Hence, for a given $R$ and $\epsilon$, either $R^{1-\epsilon}T^D(R)$ is decreasing or the following inequality holds: 
\begin{align}
&\hspace*{-2.5cm}\!\!\!
\text{\bf Inequality  }  X(\epsilon,R) :  \hspace{1cm}   
\int_{B_R(0)^c}(\epsilon+\frac{R}{r}) |F_A^D|^2 dv <6T^B(R).\nonumber
\end{align}
%\begin{align}\label{eq:XeR}
%\tag{Ineq.~$X(\epsilon,R)$} \,\,\int_{B_R(0)^c}(\epsilon+\frac{R}{r}) |F_A^D|^2 dv <6T^B(R)
%.\end{align}

\section{Estimating \texorpdfstring{$T(R)$}{T(R)}}

\begin{proposition}\label{proppart}
For every $\epsilon>0$,  there is $A_\epsilon>0$ such that 
\begin{align}
T(R)\leq A_\epsilon R^{\epsilon - 1}.
\end{align}
\end{proposition}
\begin{proof}
Let $R$ be large. Then \eqref{offdiag2} gives the estimate 
\begin{align}\label{offdiag3}
 T^B(R)&\leq   4\delta_{min}^{-2}C_n (\frac{T(\frac{R}{2})}{ R^2}+ T(\frac{R}{2})^{3-2(\frac{3}{4})^{n}}).
\end{align}  
 If for some $ \epsilon\in(0,1)$, inequality $X(\epsilon,R)$ holds, then
\begin{align}\label{crudeX}
 T(R)& \leq   28\epsilon^{-1}\delta_{min}^{-2}C_n (\frac{1}{ R^2}+ T(\frac{R}{2})^{2-2(\frac{3}{4})^{n}})T(\frac{R}{2}).
\end{align}
Hence 
\begin{align}\label{crudeX2}
 R^{1-\epsilon}T(R)& \leq    28\epsilon^{-1}C_n 2^{1-\epsilon}\delta_{min}^{-2}(  \frac{1}{R^{2}}+ T(\frac{R}{2})^{2-2(\frac{3}{4})^{n}})(\frac{R}{2})^{1-\epsilon}T(\frac{R}{2}).
\end{align}
There exists $R_{n,\epsilon}$ so that for $R\geq R_{n,\epsilon},$   $\frac{28C_n 2^{1-\epsilon}}{\epsilon\delta_{min}^{2}}(  \frac{1}{  R^{2}}+
T(\frac{R}{2})^{2-2(\frac{3}{4})^{n}})\leq 1$, and therefore, if  inequality $X(\epsilon,R)$ holds for some $R\geq R_{n,\epsilon}$, we have 
\begin{align}\label{step1}
R^{1-\epsilon}T(R)& \leq   (\frac{R}{2})^{1-\epsilon}T(\frac{R}{2}).
\end{align} 
If inequality $X(\epsilon,\frac{R}{2})$ holds, we can iterate \eqref{step1}.  
If we can iterate until $\frac{R}{2^k}\leq R_{n,\epsilon}$, we obtain the desired estimate. So assume the iteration stops for some $\frac{R}{2^k}>R_{n,\epsilon}$. Thus, $X(\epsilon,\frac{R}{2^k})$ fails for $\frac{R}{2^k}>R_{n,\epsilon}$. %  (and $(\frac{R}{2^k})^{1-\epsilon}T(\frac{R}{2^k})>(\frac{R}{2^{k+1}})^{1-\epsilon}T(\frac{R}{2^{k+1}})$).
 Let $[a,\frac{R}{2^k}]$ denote the largest interval bounded above by $\frac{R}{2^k}$ on which $X(\epsilon,s)$ fails. Then  $s^{1-\epsilon}T^D(s)$ is monotonically decreasing and $T(s)\leq \frac{7+\epsilon}{6}T^D(s)$ on $[a,\frac{R}{2^k}]$. Hence, if $a\leq R_{n,\epsilon}$, we obtain the desired estimate. If $a>R_{n,\epsilon}$,  we have:
\begin{align}\label{dev1}
  R^{1-\epsilon}T(R)&\leq (\frac{R}{2^k})^{1-\epsilon}T(\frac{R}{2^k})\leq \frac{7+\epsilon}{6}a^{1-\epsilon}T^D(a) \leq (\frac{7+\epsilon}{6})^2a^{1-\epsilon}T(a) \nonumber\\
&\leq (\frac{7+\epsilon}{6})^2(\frac{a}{2})^{1-\epsilon}T(\frac{a}{2}).
\end{align}
For later use, we also note that for  $a>R_{n,\epsilon}$, the inequality $X(\epsilon,a)$ is saturated and therefore $T(a)\leq \frac{7}{\epsilon}T^B(a)$ giving the additional estimate 
\begin{align}\label{intermed}
T(R) \leq \frac{7}{\epsilon}(\frac{7+\epsilon}{6})^2 R^{\epsilon-1}a^{1-\epsilon}T^B(a).
\end{align}
We now iterate \eqref{dev1}, replacing $R$ with $\frac{a}{2}$. Unfortunately, each iteration introduces a factor of $(\frac{7+\epsilon}{6})^2$. Since $\frac{a}{2}\leq \frac{R}{2}$, we reach the interval $[\frac{R_{n,\epsilon}}{2},2R_{n,\epsilon}]$ after at most $\ln_2(R)$ iterations. Hence  there is a constant $c_{n,\epsilon}>0$ depending on $n$ and $\epsilon$ so that 
\begin{align}\label{corr1}
& R^{1-\epsilon}T(R)\leq  (\frac{7+\epsilon}{6})^{2\ln_2(R)} c_{n,\epsilon} = R^{2\ln_2(\frac{7+\epsilon}{6}) } c_{n,\epsilon}.
\end{align} 
Since $2\ln_2(\frac{7+\epsilon}{6}) %= \frac{2\ln (1+\frac{1+\epsilon}{6})}{\ln(2)}\leq \frac{   1+\epsilon }{3\ln(2)}
\leq  \frac{   1+\epsilon }{2},$ \eqref{corr1} yields for all large $R$,
\begin{align}\label{corr2}
& R^{\frac{1}{2}-\frac{3\epsilon}{4}}T(R)\leq     c_{n,\epsilon}.
\end{align}
Choosing $\epsilon=\frac{1}{9}$, we deduce $T(R)\leq c_{n,\epsilon}R^{-\frac{1}{3}}.$ Choosing $n\geq 3$, \eqref{offdiag3} implies $T^B(R) =  O(R^{-\frac{2}{3}}).$  Now apply \eqref{intermed} to   deduce that $T(R) = O(R^{-\frac{2}{3}}).$ In general, if we have $T(R) = O(R^{-p}),$ $p<1$, then by \eqref{offdiag3}, $T^B(R) = O(R^{-2p}).$ By \eqref{intermed}, $T(R) = O( R^{ -2p}+R^{\epsilon-1})$. Iterate until we deduce $T(R) = O(R^{\epsilon-1})$, as claimed.

%\begin{align}
%&(\frac{R}{2^{k+1}})^{1-\epsilon}T(\frac{R}{2^{k+1}})<(\frac{R}{2^k})^{1-\epsilon}T(\frac{R}{2^k})\leq a^{1-\epsilon}T^D(a) +(\frac{R}{2^k})^{1-\epsilon}T^B(\frac{R}{2^k})\nonumber\\
%&\leq a^{1-\epsilon}T^D(a) +(\frac{R}{2^k})^{1-\epsilon} 4\delta_{min}^{-2}C_n (\frac{4^k}{ R^2}+ T(\frac{R}{2^{k+1}})^{2-2(\frac{3}{4})^{n}})T(\frac{R}{2^{k+1}})\nonumber\\
%&\leq a^{1-\epsilon}T^D(a) + 2^{1-\epsilon}4\delta_{min}^{-2}C_n (\frac{4^k}{ R^2}+ T(\frac{R}{2^{k+1}})^{2-2(\frac{3}{4})^{n}})(\frac{R}{2^{k+1}})^{1-\epsilon}T(\frac{R}{2^{k+1}})\nonumber\\
%\end{align}

%Now we can extend the monotonicity relations for the discrete values $\frac{a}{2^j}$ until $X(\epsilon,R)$ fails for some $\frac{a}{2^j}>R_{n,\epsilon}$. Iterating, we deduce 
%$$R^{1-\epsilon}T(R) \leq   L^{1-\epsilon}T(L),$$
%for some $L\leq 2R_{n,\epsilon}$, and the proposition follows. 
\end{proof}

\begin{theorem}\label{goodtheorem}
For $\epsilon\in (0,1),$ there exist  $C_1,C_{2,\epsilon}>0$, such that $T(R)\leq C_1R^{-1},$ and $T^B(R)\leq C_{2,\epsilon}R^{-3+\epsilon}.$
\end{theorem}
\begin{proof}
The estimate for $T^B$ is an immediate consequence of  Proposition \ref{proppart} and \eqref{offdiag3}.
Combining this estimate with \eqref{dortmund} yields  
\begin{align}\label{nirvana}
\frac{d}{dR}(RT^D(R)) \leq 6C_{2,\epsilon}R^{\epsilon-3}.
\end{align}
Hence $RT^D(R)$ and $RT(R)$ are bounded. This implies $T(R)\leq C_1R^{-1},$ for some $C_1>0$. 
\end{proof}
\begin{corollary}
$$S(R)=O(R^{-2}).$$
\end{corollary}
\begin{proof}
Let $p\in S_R$, $R$ large. Choosing $\delta = 1$ and $L=\frac{R}{2}$ in  Lemma \ref{moseri}, we see that there exists $C>0$ such that 
\begin{align}\label{hahn1}
R^2S(R)\leq  C \max\{1,[(\frac{R}{2})^2S(\frac{R}{2})]^{\frac{3}{4}}\}.
\end{align}
Here we have used $RT(R)$ is bounded, $\forall R>R_0.$ Iterating this estimate, we have 
\begin{align}\label{hahn2}
R^2S(R)\leq  C^{\sum_{j=0}^k(\frac{3}{4})^j} \max\{1,[(\frac{R}{2^{k+1}})^2S(\frac{R}{2^{k+1}})]^{(\frac{3}{4})^k}\}.
\end{align} 
Terminate the iteration when the max on the right hand side is $1$ or when $\frac{R}{2^{k+1}}\leq 2R_{00}$ to obtain the claimed result. 
\end{proof}
\begin{corollary}
$$|\Phi_\lambda|= O(R^{-1}).$$
\end{corollary}
\begin{proof}
Apply the fundamental theorem of calculus, using $|\Phi_\lambda(p)|\to 0$ as $|p|\to \infty.$ 
\end{proof}
Let $ad$ denote the adjoint represesentation. 
\begin{corollary}\label{grad1}
$\|R^{2+m}\nabla^m F_A\|_{L^\infty}+\|R^{2+m}ad^m(\Phi)^m F_A\|_{L^\infty}<\infty.$
\end{corollary}
\begin{proof}
 Let $\eta$ be the cutoff function defined on page \pageref{page:eta} in the paragraph preceding \eqref{offdiag1}. 
For $|p|$ sufficiently large, we have 
$$\int_{\IR^3}\eta(\frac{2|x-p|}{|p|})\Delta\frac{1}{2}|F_A|^2dv = \int_{\IR^3}\eta(\frac{2|x-p|}{|p|})(-|\nabla F_A|^2-|[\Phi,F_A]|^2+3\langle [F_{jk},F_{km}],F_{jm}\rangle)dv.$$
Integrating by parts and using $|\Delta \eta(\frac{2|x-p|}{|p|})| = O(|p|^{-2})$ implies
$$\int_{B_{\frac{|p|}{4}}(p)}( |\nabla F_A|^2+|[\Phi,F_A]|^2) dv= O(R^{-3}).$$
 Now a simple Moser iteration argument gives the desired result for  $m=1$. Induction on $m$ then gives the result for general $m$.  See, for example \cite[Proposition 16]{First}.
\end{proof}
\begin{corollary}\label{Best1}
$$T^B(R) = O(\delta^{-2m}R^{-1-2m}), \forall m,$$ 
and
$$\|r^{m} F_A^B\|_{L^\infty}<\infty, \forall m.$$ 
\end{corollary}

\section{Angular Estimate}\label{angsec}
In this section we refine our curvature estimates.  
\begin{proposition}\label{exper1}
$$\tau(L):={\frac{1}{L}+}\int_{B_L(0)^c} \frac{1}{r}|\nabla (r\Phi_{\lambda})|^2  dv<\infty,$$ and for some $c_\Phi>0$, 
$$|\nabla (r\Phi_{\lambda})|^2\leq \frac{c_\Phi^2\tau(\frac{r}{2})}{r^2} = o(r^{-2}).$$
Moreover, 
$$\hat\tau(L):= \int_{ B_L(0)^c} r|\nabla(\nabla_r(r\Phi_{\lambda}))|^2dv+\tau(L)<\infty.$$
\end{proposition}
\begin{proof}
 \begin{align}\label{boeing}
 \Delta\frac{1}{2}|r\Phi_{\lambda}|^2= -|\nabla(r\Phi_{\lambda})|^2-\frac{1}{r}\frac{\p}{\p r}|r\Phi_{\lambda}|^2  +O(r^2|\Phi_{\lambda}||\nabla^*\nabla\Phi_{\lambda}|).
  \end{align}
Multiply \eqref{boeing} by $r^{-1}$ and integrate over $B_L(0)\setminus B_R(0)$, for $L>R$,  to get 
\begin{align}\label{panam}
&L^2\frac{\p}{\p L}L^{-3}\int_{S_L(0) }  \frac{1}{2}  |r\Phi_{\lambda}|^2 d\sigma  
 \nonumber\\
&= \int_{B_L(0)\setminus B_R(0)}r^{-1}[ |\nabla (r\Phi_{\lambda})|^2 +O( r^{2}|\Phi_{\lambda}||\nabla^*\nabla\Phi_{\lambda}|)]dv\nonumber\\
&\quad +R^2\frac{\p}{\p R}R^{-3}\int_{S_R(0) }  \frac{1}{2}  |r\Phi_{\lambda}|^2 d\sigma .
\end{align}

Since $L^{-3}\int_{S_L(0) } \frac{1}{2}|r\Phi_{\lambda}|^2d\sigma = O(L^{-1})$, the left-hand side of \eqref{panam} is negative for an unbounded sequence of $L$ values. Hence, as $L\to\infty$,  
$ \int_{B_L(0)\setminus B_R(0)}[\frac{1}{r}|\nabla (r\Phi_{\lambda})|^2   +O( |\nabla^*\nabla\Phi_{\lambda}|)]dv$ is uniformly bounded,  and we can take the limit as $R\to\infty$ to deduce 
$$  \int_{B_R(0)^c} \frac{1}{r}|\nabla (r\Phi_{\lambda})|^2  dv<\infty.$$
We also have 
\begin{align}\label{autumn}
\Delta\frac{1}{2}|\nabla (r\Phi_{\lambda} )|^2=&-|\nabla\nabla (r\Phi_{\lambda} )|^2-\langle\nabla_j\nabla_j\nabla_m (r\Phi_{\lambda} ),\nabla_m(r\Phi_{\lambda} )\rangle\nonumber\\
=&-|\nabla\nabla (r\Phi_{\lambda} )|^2+\langle \nabla r\nabla^*\nabla   \Phi_{\lambda}  ,\nabla(r\Phi_{\lambda} )\rangle
- r^{-1}\nabla_r | \nabla (r\Phi_{\lambda}) |^2
\nonumber\\
&
+ \frac{4}{r^2}| \nabla_r  (r\Phi_{\lambda}) |^2
-2\langle  [F_{mr}, \Phi_{\lambda} ] ,\nabla_m(r\Phi_{\lambda} )\rangle
\nonumber\\
&
-\frac{2}{r^2}|\nabla ( r\Phi_{\lambda})|^2 -2\langle[F_{jm},  \nabla_j(r\Phi_{\lambda} )],\nabla_m(r\Phi_{\lambda} )\rangle
 \nonumber\\
\leq & \frac{c}{r^2}| \nabla (r\Phi_{\lambda}) |^2 + O(r^{-N}),
\end{align} 
for any $N$ and for some $c>0$. We now apply Moser iteration again. In Proposition \ref{whorse} choose $f:= \frac{1}{2}|\nabla (r\Phi_{\lambda} )|^2+r^{-3}$, $R_2= \frac{|p|}{2}$, and $R_1=0$.  With these choices, the inequality   \eqref{autumn} for  $N=5$ implies a uniform upper bound for the quantity $W$ defined in Proposition \ref{whorse}. Hence for some constants $c',c_\Phi>0$, 
\begin{align}
| \nabla (r\Phi_{\lambda}) |^2(p)\leq c'(|p|^{-3}\int_{B_{\frac{|p|}{2}}(x)}| \nabla (r\Phi_{\lambda}) |^2dv +  |p|^{-3})\leq c_\Phi^2|p|^{-2}\tau(\frac{|p|}{2}).
\end{align} 
Multiply \eqref{autumn} by $r$ and integrate over $B_L(0)\setminus B_R(0)$, for some $L>R$ to get 
\begin{align}\label{rake}
&L^6 \frac{\p}{\p L}L^{-5}\left(\int_{S_L(0)} \frac{1}{2}|\nabla (r\Phi_{\lambda} )|^2d\sigma\right)-R^6\frac{\p}{\p R}R^{-5}\left(\int_{S_R(0)} \frac{1}{2}|\nabla (r\Phi_{\lambda} )|^2d\sigma\right)\nonumber\\
&=\int_{B_L(0)\setminus B_R(0)} [r|\nabla\nabla (r\Phi_{\lambda} )|^2-r\langle \nabla r\nabla^*\nabla   \Phi_{\lambda}  ,\nabla(r\Phi_{\lambda} )\rangle 
- \frac{9}{2r} | \nabla (r\Phi_{\lambda}) |^2
\nonumber\\
&\quad 
+2r\langle  [F_{mr}, \Phi_{\lambda} ] ,\nabla_m(r\Phi_{\lambda} )\rangle
 +2r\langle[F_{jm},  \nabla_j(r\Phi_{\lambda} )],\nabla_m(r\Phi_{\lambda} )\rangle]dv.
\end{align}
The finiteness of $\hat \tau(L)$ now follows from the same argument as for $\tau(L)$, except that it now uses $\tau(L)<\infty$ as an input. 
\end{proof}

\begin{corollary}\label{shortcut}
Let $c_\Phi>0$ be the constant introduced in Proposition \ref{exper1}.
\begin{align}
  |F_A(\frac{\p}{\p r},\cdot)|=O(r^{-2}\tau^{\frac{1}{2}}(\frac{r}{2})),
\end{align}
\begin{align}\label{treefrog}
\Phi_\lambda = -r\ast F_A^\lambda(\frac{\p}{\p r}) + O(r^{-1}\tau^{\frac{1}{2}}(\frac{r}{2})),
\end{align}
and 
\begin{align}\label{roundabout}
|\nabla^0\Phi_\lambda| \leq  c_\Phi r^{-2}\tau^{\frac{1}{2}}(\frac{r}{2}), 
\end{align}
where we have written $\nabla = dr\otimes \nabla_{\frac{\p}{\p r}}+\nabla^0.$
\end{corollary}
\begin{proof}
Combine the pointwise bound for $|\nabla r\Phi_\lambda|$ from Proposition \ref{exper1} with the approximate monopole equation \eqref{lambdamon} and the $|F^B_A|$ estimate of Corollary \ref{Best1}. 
\end{proof}

\section{Connection Asymptotics}  
For $i\lambda\in \spec (\Phi(\infty)),$ set $E_\lambda:= P_\lambda E.$ By Corollary \ref{spec0}, the decomposition $E =\oplus E_\lambda$ is well defined outside a compact set. 
In this section,  we   show that $E$ further decomposes as $\oplus_{\lambda,\beta}E(\lambda,\beta)$, and 
$\Phi$ acts on $E(\lambda,\beta)$ as $i (\lambda+\frac{\beta}{r})+o(r^{-1}).$

Let $d_\lambda:= \rank(P_\lambda)$. For large  $p$, denote the eigenvalues of  $\Phi_\lambda$ (restricted to $\Im (P_\lambda)$) by $\{\frac{ia_j(p,\lambda)}{r} \}_{j=1}^{d_\lambda}$, $a_j(p,\lambda)\leq a_{j+1}(p,\lambda)$. For every $\alpha>0$, there is a maximal decomposition of $\{1,\ldots,d_\lambda\}$ into disjoint sets $\{1,\ldots,d_\lambda\}=\cup_{i=1}^{n(\alpha,p,\lambda)}J_i^\alpha(p,\lambda)$, where $a_j(p,\lambda)-a_l(p,\lambda)> 4 \alpha$, if $l\in J_i^\alpha(p,\lambda)$ and $j\in  J_{i+1}^\alpha(p,\lambda)$. Here `maximal' means that  $n(\alpha,p,\lambda)$ is maximal.  Let 
$$\delta_\alpha(r):= \min\{a_j(q,\lambda)-a_l(q,\lambda):l\in 
J_i^\alpha(p,\lambda), j\in  J_{i+1}^\alpha(p,\lambda), 1\leq i<n(\alpha,p,\lambda),  
\text{ and } |q|=r\}.$$

From  Corollary  \ref{shortcut} and the fundamental theorem of calculus, we see that  for $|p| = |q| = R$, $|a_j(p,\lambda)-a_j(q,\lambda)|\leq \pi c_\Phi\tau^{\frac{1}{2}}(\frac{R}{2})$. %, and $|F_A(\frac{\p}{\p r},\cdot)|=O( \tau^{\frac{1}{2}}(\frac{R}{2})R^{-2})$. 
Given $\alpha$, we restrict attention to $p$ sufficiently large so that 
$\pi c_\Phi\tau^{\frac{1}{2}}(\frac{|p|}{2})<\frac{\alpha}{4}$. 
Let $C(J_k^\alpha(p,\lambda))$ be a curve surrounding $\{\frac{a_j(p,\lambda)}{r}\}_{j\in J_i^\alpha(p,\lambda)}$, with  distance $\geq \frac{\alpha}{2r}$  between $C(J_k^\alpha(p,\lambda))$ and $\{\frac{ia_j(p,\lambda)}{r} \}_{j=1}^{d_\lambda}$. Assume $C(J_i^\alpha(p,\lambda))$ does not enclose any $\frac{a_j(p,\lambda)}{r}$ for $j\not\in J_i^\alpha(p,\lambda).$ On $S_R(0)$, set 
$$P_{J_i^\alpha(p,\lambda)} := \frac{1}{2\pi i}\int_{C(J_i^\alpha(p,\lambda))}(z+i\lambda-\Phi)^{-1}dz,$$
$$\Phi_{J_i^\alpha(p,\lambda)} := P_{J_i^\alpha(p,\lambda)}\Phi_\lambda P_{J_k^\alpha(p,\lambda)}=\frac{1}{2\pi i}\int_{C(J_i^\alpha(p,\lambda))}z(z+i\lambda-\Phi)^{-1}dz,$$
and
$$F^{J_i^\alpha(p,\lambda)}_A := P_{J_i^\alpha(p,\lambda)}F_AP_{J_i^\alpha(p,\lambda)}.$$
These constructions extend to any neighborhood of $S_R(0)$ on which, say, $\delta_\alpha(r)>2\alpha$. Let $N_\alpha S_R(0)$ be such a neighborhood. 
\begin{proposition}\label{lunch}
In $N_\alpha S_R(0)$, for $\lambda_1\not = \lambda_2$, 
\begin{align}
\label{uneqlambda}|P_{J_i^\alpha(p,\lambda_1)}d_A P_{J_l^\alpha(p,\lambda_2)}|= O(r^{-N}), \forall N,
\end{align}
and for $i\not = j$ 
\begin{align}
\label{uneqbeta}|P_{J_i^\alpha(p,\lambda)}d_A P_{J_j^\alpha(p,\lambda)}|= o(\delta_\alpha(|q|)^{-1}r^{-1}).
\end{align} 
\end{proposition}
\begin{proof}
Compute for $J_i^\alpha(p,\lambda_1)\not = J_j^\alpha(p,\lambda_2)$,
\begin{align}\label{allcoms}
 &P_{J_i^\alpha(p,\lambda_1)} d_A  P_{J_j^\alpha(p,\lambda_2)} = P_{J_i^\alpha(p,\lambda_1)} [d_A , P_{J_j^\alpha(p,\lambda_2)}]\nonumber\\
=& P_{J_i^\alpha(p,\lambda_1)}\frac{(\ast F_A)_a^b}{i\lambda_a-i\lambda_b} w_b\otimes w_a^\dagger P_{J_j^\alpha(p,\lambda_2)}.
\end{align}
where $\lambda_a(x) = \lambda_2+ \frac{a_{j_a}(q)}{r}$ and $\lambda_b(x) = \lambda_1+ \frac{a_{j_b}(q)}{r},$ for some $a_{j_a}(q),a_{j_b}(q)$ satisfying 
$|a_{j_a}(q)-a_{j_b}(q)|\geq \delta_\alpha(|q|).$ 

By Corollary \ref{Best1}, $|F_A^B|= O(r^{-N}),$ for every $N$. Hence \eqref{uneqlambda} follows from \eqref{allcoms}. 
For $\lambda_1(x)  = \lambda_2=\lambda$, $i\not = j$, we have 
$|\frac{(\ast F_A)_a^b}{\lambda_a-\lambda_b}|\leq C|q||P_{J_i^\alpha(p,\lambda )} F_A P_{J_j^\alpha(p,\lambda)}|.$ By  Equation \eqref{treefrog}, 
$|P_{J_i^\alpha(p,\lambda )} F_A P_{J_j^\alpha(p,\lambda)}| = o(r^{-2})$, and \eqref{uneqbeta} follows. 
\end{proof}
\begin{lemma}\label{aboveanalysis}
For $q\in N_\alpha S_R(0)$, 
$$\tr   \Phi_{J_k^{\alpha}(p)}(q)=     \frac{i}{2|q|}c_1(E_{J^\alpha_k(p,\lambda)})(S^2)+ O(|q|^{-1}\tau^{\frac{1}{2}}(\frac{|q|}{2})[1+\alpha^{-2}|q|^{-2}]. 
$$
\end{lemma}
\begin{proof}
 Pre- and post-composing \eqref{treefrog} with $P_{J_k^{\alpha}(p)}$, we have 
\begin{align}\label{crapaud}
\Phi_{J_k^{\alpha}(p)} = -r\ast F_A^{J_k^{\alpha}(p)}(\frac{\p}{\p r}) + O(R^{-1}\tau^{\frac{1}{2}}(\frac{R}{2})).
\end{align}
Let $E_{J_k^\alpha(p,\lambda)}:= P_{J_k^\alpha(p,\lambda)}E$. Let $d_{J_k^\alpha(p,\lambda)}$ denote the rank of $E_{J_k^\alpha(p,\lambda)}$.  This subbundle is well defined in $N_\alpha S_R(0)$. It has the induced connection $P_{J_k^\alpha(p,\lambda)}d_AP_{J_k^\alpha(p,\lambda)},$ with curvature  
\begin{align}
F_{J_k^\alpha(p,\lambda)}:&= P_{J_k^\alpha(p,\lambda)}d_AP_{J_k^\alpha(p,\lambda)}d_AP_{J_k^\alpha(p,\lambda)}\nonumber \\
&= 
 F_A^{J_k^\alpha(p,\lambda)}+P_{J_k^\alpha(p,\lambda)}[d_A,P_{J_k^\alpha(p,\lambda)}][d_A,P_{J_k^\alpha(p,\lambda)}]P_{J_k^\alpha(p,\lambda)}\nonumber\\
&=  F_A^{J_k^\alpha(p,\lambda)}+P_{J_k^\alpha(p,\lambda)}  \frac{w_b\otimes w_a^\dagger}{\lambda_a-\lambda_b}  (I-P_{J_k^\alpha(p,\lambda)}) (\ast F_A)_a^b\wedge (\ast F_A)_c^a\frac{w_a\otimes w_c^\dagger}{\lambda_c-\lambda_a}  P_{J_k^\alpha(p,\lambda)}.
\end{align}
 
Taking traces, we have
\begin{align}\label{possum}
\tr  F_{J_k^\alpha(p,\lambda)}  &= \tr  F_A^{J_k^\alpha(p,\lambda)}-\sum_{b\in J_k^\alpha(p,\lambda)}\sum_{ i\lambda_a\not\in \spec (\Phi_{J_k^\alpha(p,\lambda)})}  \frac{(\ast F_A)_a^b\wedge (\ast F_A)_b^a}{(\lambda_a-\lambda_b)^2}    \nonumber\\
&=\tr  F_A^{J_k^\alpha(p,\lambda)}+O(\tau^{\frac{1}{2}}(\frac{R}{2})\alpha^{-2}R^{-2}).
\end{align}
Here we have used Corollary \ref{shortcut} (and the vanishing of the $dr\wedge dr$ component of $\ast F_A\wedge \ast F_A$) to gain the $\tau^{\frac{1}{2}}(\frac{R}{2})$ factor in the error estimate for \eqref{possum}.

Taking traces and integrating yields 
\begin{align}
R^{-1}\int_{S_R(0)}\tr  i\Phi_{J_k^{\alpha}(p)}d\sigma &=  \int_{S_R(0)}\tr  iF^{J_k^{\alpha}(p)}  + O( \tau^{\frac{1}{2}}(\frac{R}{2}))\nonumber\\
&=  -2\pi c_1(E_{J^\alpha_k(p,\lambda)})(S^2)+ O(\tau^{\frac{1}{2}}(\frac{R}{2})\alpha^{-2} ).
\end{align}
 Hence, $\forall q\in N_\alpha S_R(0)$, 
\begin{align}\label{vix}
 \tr   \Phi_{J_k^{\alpha}(p)}(q)=     \frac{i}{2R}c_1(E_{J^\alpha_k(p,\lambda)})(S^2)+ O(R^{-1}\tau^{\frac{1}{2}}(\frac{R}{2})[1+\alpha^{-2}R^{-2}]), 
\end{align}
as desired.  
\end{proof}
\begin{theorem}\label{specf}
For each $i\lambda\in \spec (\Phi(\infty))$, the spectrum of $|p|\Phi_\lambda(p)$ has a well defined limit as $|p|\to \infty$. At large radius, $E$ decomposes into a direct sum of vector bundles $E=\oplus E(\lambda,\beta)$, such that $\Phi$ acts as $i \lambda+ i \frac{\beta}{r}+o(r)$ on $E(\lambda,\beta)$, with $\beta\in \frac{1}{2}\IZ$.  
\end{theorem}
\begin{proof}
Fix $i\lambda\in \spec (\Phi(\infty))$. It suffices to prove that there exists $\alpha_\infty>0$ and a decomposition $\{1,\ldots,d_\lambda\}=\cup_iJ_i^{\alpha_\infty}(\infty)$ and $R(\alpha_\infty)$ so that
\begin{enumerate}
\item[(i)] For each $i$, whenever  $j\in J_i^{\alpha_\infty}(\infty)$ and $l\not \in J_i^{\alpha_\infty}(\infty)$,   $|a_j(p,\lambda)-a_l(p,\lambda)|>4\alpha_\infty$, for all $p\in B_{R(\alpha_\infty)}(0)^c$. 
\item[(ii)]
For all $i$ and for all $j,m\in   J_i^{\alpha_\infty}(\infty),$ 
$\lim_{|p|\to\infty}|a_j(p,\lambda)-a_l(p,\lambda)|=0.$
\item[(iii)] $\frac{c_1(E(\lambda,\beta))}{ \rank (E(\lambda,\beta))}\in \IZ.$
\end{enumerate}

We now prove (i) and (ii). Consider  $|p|$ and $R$ large.  
Suppose that $\forall q\in S_R(0)$ and for some $j\in J_i^{\alpha}(p,\lambda)$, $a_{j+1}(q,\lambda)-a_{j}(q,\lambda)> \beta,$ with $\beta>R^{-1}$. Let $J':= J_i^{\alpha}(p,\lambda)\cap [0,j]$ and $J'' =J_k^{\alpha}(p,\lambda)\setminus J'.$ 
Then   Lemma \ref{aboveanalysis} implies  
\begin{align*} 
 \frac{1}{2i \pi}\tr   \Phi_{J'}(q)=     c_1(E_{J'})(S^2)+ O(\tau^{\frac{1}{2}}(\frac{R}{2})),%  + O(\hat\tau(R)^{\frac{1}{2}}),
\end{align*}
and
\begin{align*} 
 \frac{1}{2i \pi}\tr   \Phi_{J''}(q)=     c_1(E_{J''})(S^2)+ O(\tau^{\frac{1}{2}}(\frac{R}{2})),%  + O(\hat\tau(R)^{\frac{1}{2}}).
\end{align*}
Suppose $a_{i+1}(q,\lambda)-a_i(q,\lambda)<\frac{1}{100d_\lambda}$ for each $i,i+1\in J^\alpha_i(p,\lambda),$ $\forall q\in S_R(0)$. Then integrality implies
\begin{align*}
c_1(E_{J'})(S^2)= \frac{|J'|}{|J^\alpha_i(p,\lambda)|}c_1(E_{J^\alpha_i(p,\lambda)})(S^2)\text{ and }c_1(E_{J''})(S^2)= \frac{|J''|}{|J^\alpha_k(p,\lambda)|}c_1(E_{J^\alpha_i(p,\lambda)})(S^2),
\end{align*}
where for any  $J\subset \{1,\ldots,d_\lambda\}$,  $|J|$ denotes its cardinality. 
Then for the distinguished index $j$ introduced above, we have 
\begin{align}
&|J''|\frac{a_{j+1}}{2  \pi}\leq c_1(E_{J''})(S^2)+O(\tau^{\frac{1}{2}}(\frac{R}{2}))\nonumber\\  
\Rightarrow& \frac{a_{j+1}}{2  \pi}\leq   \frac{c_1(E_{J^\alpha_i(p,\lambda)})(S^2)}{|J^\alpha_i(p,\lambda)|}+O(\tau^{\frac{1}{2}}(\frac{R}{2})). 
\end{align} 
Similarly $$\frac{a_{j}}{2  \pi}\geq   \frac{c_1(E_{J^\alpha_i(p,\lambda)})(S^2)}{|J^\alpha_i(p,\lambda)|}+O(\tau^{\frac{1}{2}}(\frac{R}{2})).$$    
These inequalities imply 
$$a_{j+1}(q,\lambda)-a_{j}(q,\lambda)= O(\tau^{\frac{1}{2}}(\frac{R}{2})).$$ 
Hence $\forall i$ and $\forall l\in J_i^{\alpha}(p,\lambda)$, 
\begin{align}\label{ii}
a_{l}(q,\lambda)=  \frac{c_1(E_{J^\alpha_i(p,\lambda)})(S^2)}{|J^\alpha_i(p,\lambda)|}+O(\tau^{\frac{1}{2}}(\frac{R}{2})).
\end{align}
Thus, for all $i$, and for $\tau^{\frac{1}{2}}(\frac{R}{2})$ sufficiently small (depending on the distances between distinct elements of the set $\cup_{j=1}^{d_\lambda}\frac{1}{j}\IZ$) the  gaps between the spectrum of $\Phi_{J^\alpha_i(p,\lambda)}$ and $\Phi_{J^\alpha_{i+1}(p,\lambda)}$ have fixed lower bounds. 
 Consequently, for any $\alpha$ sufficiently small, there exists $R(\alpha)$ large so that for $|p|>R(\alpha)$,  $E_{J^\alpha_m(p,\lambda)}$ is defined on $B_{R(\alpha)}^c(0)$, $\forall m$, and with a uniform lower bound in the gaps in the spectrum between $r\Phi_{J^\alpha_m(p,\lambda)}(q)$ and $r\Phi_{J^\alpha_j(p,\lambda)}(q)$ for $j\not = m$, and $|q|\geq R(\alpha)$. This establishes condition (i). Condition (ii) immediately follows from \eqref{ii}.

To prove (iii), note that by Corollary \ref{shortcut}, the connection $A$ restricted to $S_R(0)$, converges as $R\to\infty$ to a Hermitian Einstein connection on $\IC\IP^1$. Hence $E(\lambda,\beta)$ is polystable, decomposing into a direct sum of line bundles with the same first Chern class. (See also \cite[Section 2]{CN22}.)
\end{proof}

\section{\texorpdfstring{$k$}{k}-centered Taub-NUT}
Let  $M$ be a $k$-centered Taub-NUT manifold, $TN_k$, with the convention that for $k=0$,  $TN_0$ is $S^1\times \IR^3$. We use the notation introduced in Section \ref{Intro} to describe the geometry of $M$. 
We orient $TN_k$ so that the volume forms on $TN_k$ and $\IR^3$ are related by
\begin{align}
 dv_M = V\Pi^*dv_{\IR^3}\wedge d\theta.
\end{align}
In particular, for a differential one-form $\phi$ annihilated by inner product with $\frac{\p}{\p \theta}$, we have 
\begin{align}\label{ast1}
\ast_M\phi = V (\ast_{\IR^3}\phi)\wedge (d\theta+\omega).
\end{align}
Henceforth we simply write $dv$ for $dv_M$. 
Let $\mathcal{R}$ denote the Riemann curvature tensor of $TN_k$.  Its norm decays cubically. 

Let $(E,A)$ be a rank $N$ hermitian vector bundle on $TN_k$, with connection whose curvature, $F_A$ 
  is square integrable and anti-self-dual.  We denote the associated covariant derivative as $\nabla^A$ or simply $\nabla$, if the connection is clear from context.   
 Set 
$$T_I(R):= \int_{\Pi^{-1}(B_R(0)^c)}|F_A|^2dv,$$
and 
$$S_I(R):= \|F_A\|_{L^\infty(\Pi^{-1}(B_R(0)^c))}.$$
Corollary \ref{coro1} extends to this context to give us $\lim_{R\to\infty}S_I(R) = 0$. Corollary \ref{moseri} and Lemma \ref{goodcoro} also extend to this context with minor changes, as we now show. Changing the symbol $\Phi$ used in \cite[Section 4.2]{First} to $\Psi$  in order to avoid confusion with the Higgs fields of preceding sections, we define for sections $\psi$ of a hermitian vector bundle over $TN_k$, 
$$\Psi(\psi)(x):= \int_{\Pi^{-1}(x)}|\psi|^2d\theta,$$ 
and compute, for some constant $c>0$, (see \cite[Eq.~(51)]{First})
\begin{align}
\Delta\frac{1}{2}|\Psi(F_A)|^2 \leq -2V|\Psi(\nabla F_A)|^2+c(\|\mathcal{R}\|_{L^\infty(\Pi^{-1}(x))}+\|F_A\|_{L^\infty(\Pi^{-1}(x))})|\Psi( F_A)|^2). 
\end{align}
This adds an $\|\mathcal{R}\|_{L^\infty(\Pi^{-1}(x))}$ term to  the estimates of Lemma \ref{moseri} and Corollary \ref{goodcoro}, but this additional $O(\frac{1}{r^3})$ term can be absorbed into the other terms. So, Corollary \ref{moseri} and Lemma \ref{goodcoro} extend to this context.

For $TN_k$  we no longer have a Higgs field $\Phi$. In the following sections, we construct a spectral truncation $\Phi^I_\lambda$ of the unbounded operator $\nabla_\theta:= \nabla_{\frac{\p}{\p\theta}}$,  analogous to $\Phi_\lambda$. We then repeat the arguments  introduced for monopoles to establish  quadratic curvature decay and to determine the asymptotic structure of the connection. As with monopoles, to define such a spectral truncation in a neighborhood of $\infty$, we must first show that the spectrum of $\nabla_\theta$ 
is asymptotically constant. In order to prove this without first having a Higgs field, we will require spectral functions analogous to $\zeta$ functions. In the next section, we recall basic facts about connections and then introduce spectral $\zeta$ functions. 

\section{Hilbert bundles and Spectral Functions}
\subsection{Geometry of Hilbert bundles}
Consider the Hilbert 
 bundle $\mathcal{L}$ over $\IR^3\setminus\{\nu_1,\ldots,\nu_k\}$  
whose fibers are $\mathcal{L}_b:= L^2(\Pi^{-1}(b),E)$. Endow this bundle with the metric $$\langle s_1,s_2\rangle_{\mathcal{L}}(b):= \int_0^{2\pi}\langle s_1,s_2\rangle(b,\theta)d\theta.$$
In particular, we have dropped the $V^{-\frac{1}{2}}$ factor that would arise from simply integrating over the fiber with respect to the Riemannian structure. 
Denote the horizontal lift of a vector field $X$ on $\IR^3$ to $TN_k$ by 
$$X^h := X - \omega(X)\frac{\p}{\p \theta}.$$
with respect to a local trivialization of the circle bundle. 
With this notation, $\mathcal{L}$ admits the metric compatible connection given by 
$$\nabla_X^{\mathcal{L}}s := \nabla_{ X^h}^As.$$
We omit the $\mathcal{L}$ superscript when no confusion should occur. The curvature $F_{\mathcal{L}}$ is given by the unbounded operator
\begin{align}
F_{\mathcal{L}}(\frac{\p}{\p x^i},\frac{\p}{\p x^j})= F_A((\frac{\p}{\p x^i})^h,(\frac{\p}{\p x^j})^h)-d\omega(\frac{\p}{\p x^i},\frac{\p}{\p x^j})\nabla^A_{ \theta}.
\end{align}
The infinite dimensional analog of the adjoint bundle of $\mathcal{L}$ is the bundle $ad(\mathcal{L})$ whose fiber at $b$ is the space of Hilbert-Schmidt operators on $\mathcal{L}_b$. This bundle is hermitian with inner product $\tr  AB^\dagger = \langle A,B\rangle$ and hermitian connection 
$$\nabla^{ad(\mathcal{L})}_XA:= [\nabla^{\mathcal{L}}_X,A],$$
and curvature 
$$F_{ad(\mathcal{L})}A := [F_{\mathcal{L}},A].$$
\subsection{Spectral functions}\label{subsecsf}
For $w>0$ and $x\in \IR^3\setminus \{\nu_1,\ldots,\nu_k\}$, set 
$$\zeta(x,w,l):= \Tr_{L^2(\Pi^{-1}(x))}(w^2-\nabla_\theta^2)^{\frac{l}{2}}.$$
Let $\mathcal{C}$ denote the two vertical lines surrounding $i\IR\subset \IC$ consisting of $(i\IR+a)\cup (i\IR-a)$, for some $a>0$, oriented counter clockwise. Then for $l\in \IN$,
\begin{align}
\zeta(x,w,2l) = \Tr_{L^2(\Pi^{-1}(x))}\frac{1}{2\pi i}\int_{\mathcal{C}}(w^2-z^2)^{l}(z-\nabla_\theta)^{-1}dz.
\end{align} 
\begin{lemma}
Let $f(z)$ be a holomorphic function on $-\delta<\Re\, z<\delta$, for some $\delta>0$. Assume  $|f(x+iy)|= O(\frac{1}{1+y^2})$  on this domain. Then 
\begin{align}\label{lemtr1}
\frac{\p}{\p x^j}\Tr\, f(\nabla_\theta) = \frac{1}{2\pi i}\Tr\, \int_\mathcal{C}f'(z)(z-\nabla_\theta)^{-1}F_A((\frac{\p}{\p x^j})^h,\frac{\p}{\p\theta}) dz,
\end{align}
and
\begin{align}\label{lemtr2}
 &\sum_j\frac{\p^2}{(\p x^j)^2}\Tr\, f(\nabla_\theta)\nonumber\\
 &=\sum_j  \frac{1}{2\pi i}\Tr\int_\mathcal{C}f'(z)(z-\nabla_\theta)^{-1}dz \left(F_A(\nabla_{(\frac{\p}{\p x^j})^h}(\frac{\p}{\p x^j})^h,\frac{\p}{\p\theta}) +F_A((\frac{\p}{\p x^j})^h,\nabla_{(\frac{\p}{\p x^j})^h}\frac{\p}{\p\theta}) \right)  \nonumber\\
&\quad +  \sum_j  \frac{1}{2\pi i}\Tr\int_\mathcal{C}f'(z) (z-\nabla_\theta)^{-1}F_A((\frac{\p}{\p x^j})^h,\frac{\p}{\p\theta})(z-\nabla_\theta)^{-1}F_A((\frac{\p}{\p x^j})^h,\frac{\p}{\p\theta}) dz.\end{align}
\end{lemma}
\begin{proof}
We use the holomorphic functional calculus to study the trace: 
\begin{align}
 \Tr\, f(\nabla_\theta) =  \frac{1}{2\pi i}\int_\mathcal{C}f(z)(z-\nabla_\theta)^{-1}dz. 
\end{align}
We first make  simple observations on the derivatives of traces. 
 Let $\phi$ be a section of $E$ defined on  $\Pi^{-1}(B_{\delta}(x))$ for some $\delta>0$ and 
 satisfying $\nabla_{(\frac{\p}{\p x^j})^h} \phi = 0$ on $\Pi^{-1} (x)$.  
\begin{align}
  &\frac{\p}{\p x^j} \int_{\Pi^{-1}(x)}\langle (z-\nabla_\theta)^{-1}\phi,\phi\rangle d\theta=  \int_{\Pi^{-1}(x)}(\frac{\p}{\p x^j})^h\langle (z-\nabla_\theta)^{-1}\phi,\phi\rangle d\theta\nonumber\\
&=  \int_{\Pi^{-1}(x)} \langle \nabla_{(\frac{\p}{\p x^j})^h}(z-\nabla_\theta)^{-1}\phi,\phi\rangle =  \int_{\Pi^{-1}(x)} \langle [\nabla_{(\frac{\p}{\p x^j})^h},(z-\nabla_\theta)^{-1}]\phi,\phi\rangle d\theta\nonumber\\
&=  \int_{\Pi^{-1}(x)} \langle (z-\nabla_\theta)^{-1} F_A((\frac{\p}{\p x^j})^h,\frac{\p}{\p\theta})(z-\nabla_\theta)^{-1}\phi,\phi\rangle d\theta.
\end{align}
From this, it is easy to see that 
\begin{align}
\frac{\p}{\p x^j}\Tr\, f(\nabla_\theta) &=  \frac{1}{2\pi i}\Tr\int_\mathcal{C}f(z)(z-\nabla_\theta)^{-1}F_A((\frac{\p}{\p x^j})^h,\frac{\p}{\p\theta})(z-\nabla_\theta)^{-1}dz\nonumber\\
&=  \frac{1}{2\pi i}\Tr\int_\mathcal{C}f(z)(z-\nabla_\theta)^{-2}F_A((\frac{\p}{\p x^j})^h,\frac{\p}{\p\theta}) dz\nonumber\\
&=  \frac{1}{2\pi i}\Tr\int_\mathcal{C}f'(z)(z-\nabla_\theta)^{-1}F_A((\frac{\p}{\p x^j})^h,\frac{\p}{\p\theta}) dz.
\end{align}
Similarly, 
\begin{align}
 &\sum_j\frac{\p^2}{(\p x^j)^2}\Tr\, f(\nabla_\theta)\nonumber\\
 &=  \sum_j  \frac{1}{2\pi i}\Tr\int_\mathcal{C}f'(z)[\nabla_{(\frac{\p}{\p x^j})^h},(z-\nabla_\theta)^{-1}F_A((\frac{\p}{\p x^j})^h,\frac{\p}{\p\theta})] dz\nonumber\\
&=  \sum_j  \frac{1}{2\pi i}\Tr\int_\mathcal{C}f'(z) (z-\nabla_\theta)^{-1}F_A((\frac{\p}{\p x^j})^h,\frac{\p}{\p\theta})(z-\nabla_\theta)^{-1}F_A((\frac{\p}{\p x^j})^h,\frac{\p}{\p\theta}) dz\nonumber\\
&+  \sum_j  \frac{1}{2\pi i}\Tr\int_\mathcal{C}f'(z)(z-\nabla_\theta)^{-1}[\nabla_{(\frac{\p}{\p x^j})^h},F_A((\frac{\p}{\p x^j})^h,\frac{\p}{\p\theta})] dz\nonumber\\
&=  \sum_j  \frac{1}{2\pi i}\Tr\int_\mathcal{C}f'(z) (z-\nabla_\theta)^{-1}F_A((\frac{\p}{\p x^j})^h,\frac{\p}{\p\theta})(z-\nabla_\theta)^{-1}F_A((\frac{\p}{\p x^j})^h,\frac{\p}{\p\theta}) dz\nonumber\\
&+  \sum_j  \frac{1}{2\pi i}\Tr\int_\mathcal{C}f'(z)(z-\nabla_\theta)^{-1}dz [F_A(\nabla_{(\frac{\p}{\p x^j})^h}(\frac{\p}{\p x^j})^h,\frac{\p}{\p\theta})+F_A((\frac{\p}{\p x^j})^h,\nabla_{(\frac{\p}{\p x^j})^h}\frac{\p}{\p\theta})].\end{align}
\end{proof}

\begin{proposition}
The spectrum of $\nabla_\theta$ on $L^2(\Pi^{-1}(p))$ converges to a limiting spectrum as $|p|\to \infty.$
\end{proposition}
\begin{proof}
Let $\{e^{-2\pi i\lambda_a(x)}\}_{a=1}^N$ denote the eigenvalues of the holonomy of $\nabla_\theta $ around $\Pi^{-1}(x)$. Let $\{s_a\}_{a=1}^N$ be a smooth frame for $E|_{\Pi^{-1}(x)}$, satisfying $\nabla_\theta s_a =i\lambda_a s_a$, with $\lambda_a\in [0,1)$. Then a unitary basis for $L^2(\frac{d\theta}{2\pi})$ sections of $E$ on $\Pi^{-1}(x)$ is given by  $\cup_{a=1}^N\{e^{ im\theta}s_a\}_{m=-\infty}^\infty.$ 
Hence, to prove convergence of the spectrum of $\nabla_\theta$, it suffices to prove the convergence of the holonomy spectrum.

Expand $F_A((\frac{\p}{\p x^j})^h,\frac{\p}{\p\theta})$ as 
$$F_A((\frac{\p}{\p x^j})^h,\frac{\p}{\p\theta}) = e^{i\mu \theta}(F_{j\theta}^\mu )_a^bs_b\langle \cdot, s_a\rangle .$$
With this notation, we rewrite \eqref{lemtr1} and \eqref{lemtr2} as

\begin{align}\label{elemtr1}
\frac{\p}{\p x^j}\Tr\, f(\nabla_\theta) &= \sum_{m,a}\frac{1}{2\pi i}  \int_\mathcal{C}f'(z)(z-i(m+\lambda_a))^{-1}dz(F_{j\theta}^0)_a^a  \nonumber\\
&= \sum_{m,a} f'(i(m+\lambda_a)) (F_{j\theta}^0)_a^a ,
\end{align}
and
\begin{align}\label{elemtr2}
 &\sum_j\frac{\p^2}{(\p x^j)^2}\Tr\, f(\nabla_\theta) \nonumber \\
 =&\sum_{j,m,a}  f'(i(m+\lambda_a)) \left(F_A^0(\nabla_{(\frac{\p}{\p x^j})^h}(\frac{\p}{\p x^j})^h,\frac{\p}{\p\theta})_a^a  +F_A^0((\frac{\p}{\p x^j})^h,\nabla_{(\frac{\p}{\p x^j})^h}\frac{\p}{\p\theta})_a^a\right)  \nonumber \\
&-  \sum_{j,a,b,m,\mu }   \frac{f'(i(m+\lambda_a))-f'(i(\mu +m+\lambda_b))}{i(\lambda_a-\lambda_b-\mu )}|(F_{j\theta}^\mu )_a^b|^2.
\end{align}
Multiply \eqref{elemtr2} by $\frac{1}{r_p}-\frac{1}{r_q}$, integrate over $B_R(0)\setminus B_L(0)$, with $L<|p|\leq |q|<R,$ and use the divergence theorem to obtain the following. 
\begin{align}\label{delemtr2.0}
&4\pi \Tr_{L^2(\Pi^{-1}(p)) }f(\nabla_\theta)-4\pi \Tr_{L^2(\Pi^{-1}(q))} f(\nabla_\theta) \nonumber \\
 &+\int_{S_R(0)} \left(\frac{r_q-r_p}{r_pr_q}\frac{\p}{\p r}+\frac{1}{r_p^2}\frac{\p r_p}{\p r}
 -\frac{1}{r_q^2}\frac{\p r_q}{\p r}\right)\Tr\, f(\nabla_\theta)d\sigma\nonumber \\
 &-\int_{S_L(0)} (\frac{r_q-r_p}{r_pr_q}\frac{\p}{\p r}+\frac{1}{r_p^2}\frac{\p r_p}{\p r}
 -\frac{1}{r_q^2}\frac{\p r_q}{\p r})\Tr\, f(\nabla_\theta)d\sigma\nonumber \\
 =&\int_{B_R(0)\setminus B_L(0)}
     \left[\sum_{j,m,a}  f'(i(m+\lambda_a)) 
     \left(F_A^0(\nabla_{(\frac{\p}{\p x^j})^h}(\frac{\p}{\p x^j})^h,\frac{\p}{\p\theta})_a^a\right.\right.\nonumber \\
   &  \hspace{5cm}\left.  +F_A^0((\frac{\p}{\p x^j})^h,\nabla_{(\frac{\p}{\p x^j})^h}\frac{\p}{\p\theta})_a^a
     \right)   \nonumber \\
 &\left. -  \sum_{j,a,b,m,\mu }   \frac{f'(i(m+\lambda_a))-f'(i(\mu +m+\lambda_b))}{i(\lambda_a-\lambda_b-\mu )}|(F_{j\theta}^\mu )_a^b|^2\right](\frac{1}{r_p}-\frac{1}{r_q})dv.
\end{align}
The $S_R(0)$ integrals are $o(1)$ as $R\to\infty$. Hence, we can take the $R\to\infty$ limit to obtain
\begin{align}\label{ldelemtr2}
 &4\pi \Tr_{L^2(\Pi^{-1}(p))} f(\nabla_\theta)-4\pi \Tr_{L^2(\Pi^{-1}(q))} f(\nabla_\theta)\nonumber \\
=&\int_{S_L(0)} (\frac{r_q-r_p}{r_pr_q}\frac{\p}{\p r}+\frac{1}{r_p^2}\frac{\p r_p}{\p r}-\frac{1}{r_q^2}\frac{\p r_q}{\p r})\Tr\, f(\nabla_\theta)d\sigma\nonumber \\
+&   \int_{ B_L(0)^c}(\frac{1}{r_p}-\frac{1}{r_q})[\sum_{j,m,a}  f'(i(m+\lambda_a) \left(F_A^0(\nabla_{(\frac{\p}{\p x^j})^h}(\frac{\p}{\p x^j})^h,\frac{\p}{\p\theta})_a^a  \right.\nonumber \\
 &\hspace{5.2cm}\left.+F_A^0((\frac{\p}{\p x^j})^h,\nabla_{(\frac{\p}{\p x^j})^h}\frac{\p}{\p\theta})_a^a \right) \nonumber \\
&-  \sum_{j,a,b,m,\mu }   \frac{f'(i(m+\lambda_a))-f'(i(\mu +m+\lambda_b))}{i(\lambda_a-\lambda_b-\mu )}|(F_{j\theta}^\mu )_a^b|^2]dv.
\end{align}
Repeating the estimates of Proposition \ref{togap} for $L$ fixed gives 
\begin{align}\label{delemtr2}
& \Tr_{L^2(\Pi^{-1}(p)) }f(\nabla_\theta)-  \Tr_{L^2(\Pi^{-1}(q))} f(\nabla_\theta) = O(\frac{1}{|p|}+T_I(\frac{|p|}{4})).
\end{align}
We note that the implicit constants in $O(\frac{1}{|p|}+T_I(\frac{|p|}{4}))$ depend on the choice of $f$. For example, if $f(z) = (u-iv+z)^{-l}$ for $0<u<1$, $v\in \IR$, $l\geq 2$, making the $u$ dependence explicit gives 
$$\Tr_{L^2(\Pi^{-1}(p)) }f(\nabla_\theta)-  \Tr_{L^2(\Pi^{-1}(q))} f(\nabla_\theta) = O(u^{-3}(\frac{1}{|p|}+T_I(\frac{|p|}{4}))).  $$
Taking the limit as $|q|\to\infty$ in \eqref{delemtr2} gives a limiting value for the trace which satisfies:
\begin{align}\label{ldelemtr2.0}
&   \lim_{|q|\to\infty}\Tr_{L^2(\Pi^{-1}(q))} f(\nabla_\theta) =  \Tr_{L^2(\Pi^{-1}(p))} f(\nabla_\theta)+O(\frac{1}{|p|}+T_I(\frac{|p|}{4})).
\end{align}
Since the traces converge to a limiting value, so too does the spectrum of $\nabla_\theta$. For finite rank operators, this is immediate. It is also easy to show in our context. 
The spectrum of $\nabla_\theta$ on $L^2(L^2(\Pi^{-1}(p))$ decomposes into $N$ (possibly with multiplicity) residue classes mod $i\IZ$. Denote these classes $\{i\bar\lambda_a(p)\}_{a=1}^N$. We correspondingly decompose 
$$\Tr_{L^2(\Pi^{-1}(p))} f(\nabla_\theta)=\sum_{a=1}^N\sum_{\lambda\in \bar\lambda_a(p)}f(i\lambda).$$
Specialize now to $f(z)= (u-iv+z)^{-2}$, with $u,v\in \IR$ and $u>0$. Choose $\mu_a(p)\in \bar\lambda_a$ that minimizes $|\lambda-v|$ for $\lambda\in \bar\lambda_a(p).$ Then we have 
\begin{align}|\sum_{\lambda\in \bar\lambda_a(p)}f(i\lambda)|&\geq   |(u+i(\mu_a-v))^{-2}|-|\sum_{m\in\IZ\setminus\{0\}}^\infty(u +i(\mu_a-v)+im)^{-2}|\nonumber\\
&\geq   (u^2+ (\mu_a-v)^2)^{-1}- \sum_{m\in\IZ\setminus\{0\}}^\infty(u^2 + (m-\frac{1}{2})^2)^{-2}\nonumber\\
&\geq   (u^2+ (\mu_a-v)^2)^{-1}-  \frac{\pi^{4}}{90} .
\end{align}
Thus we can determine the number of $\mu_a$ near $v$, for any  $v$, for $R$ sufficiently large (depending on $u$). Therefore  the convergence of the traces implies the convergence of the $i\bar\lambda_a(p)$ and hence the convergence of the spectrum of $\nabla_\theta.$ 
\end{proof}
Write $\spec (\nabla_\theta(\infty))$ for the limiting spectrum. 
\section{Spectral Truncation for Instantons}
Having achieved spectral convergence, we now define well behaved spectral truncations. Let $4\delta$ denote the smallest distance  between 2 distinct elements   of $\spec(\nabla_\theta(\infty))$. Let $i\lambda\in \spec (\nabla_\theta(\infty)).$ Then there exists $R_0>0$, so that for $|x|\geq R_0$, $-i\nabla_\theta$ has no spectrum in 
 $\mathop{\cup}\limits_{i\lambda \in \spec(\nabla_\theta(\infty))} \left(  [\lambda- 2\delta , \lambda- \delta ]\cup [\lambda+  \delta  , \lambda+ 2\delta]\right).$   

Given $i\lambda\in Spec(\nabla_\theta(\infty))$, define 
$$P_\lambda^I:= \frac{1}{2\pi i}\int_{|z|=\delta}(z+i\lambda-\nabla_\theta)^{-1}dz,$$
$$\Phi^I_\lambda:= \frac{1}{2\pi i}\int_{|z|=\delta}z(z+i\lambda-\nabla_\theta)^{-1}dz,$$
$$F^\lambda :=P^I_\lambda F_AP^I_\lambda.$$
Summing over $i\lambda\in \spec (\nabla_\theta(\infty)),$ define 
$$\Phi^I_D:=\sum_\lambda\Phi^I_\lambda,$$
$$F^D:= \sum_\lambda F^\lambda, \text{ and }F^B = F_A-F^D.$$
$$T_I^D(R):= \int_{B_R^c(0)}|F^D|^2dv \text{ and }T_I^B(R):= \int_{B_R^c(0)}|F^B|^2dv.$$
We now detail and exploit the analogy between $\Phi^I_\lambda$ in the instanton case and $\Phi_\lambda$ in the monopole case. 
The proof of \eqref{offdiag2} extends without modification to instantons, giving for $R_2>R_1>\frac{3R_2}{4}$, and $\frac{R_2}{2}>R_{00}$, that $\forall n\in \IN$, $\exists C_n>0$ such that 
\begin{align}\label{offdiagI}
 T^B_I(R_2)&\leq  \delta^{-2}64(\frac{1}{(R_2-R_1)^2}+S_I(R_1))T_I(R_1)\\
&\leq  \delta^{-2}C_n(\frac{1}{ (R_2-R_1)^2}+S_I(\frac{R_2}{2})^{(\frac{3}{4})^{n}}T_I(\frac{R_2}{2})^{2-2(\frac{3}{4})^{n}})T_I(R_1).
\end{align}  
Compute
\begin{align*}
&[\nabla_{(\frac{\p}{\p x^j})^h},\Phi^I_\lambda]  \\
= &\frac{1}{2\pi i} \int_{|z|=\delta}z(z+i\lambda-\nabla_\theta)^{-1}[\nabla_{(\frac{\p}{\p x^j})^h},\nabla_\theta](z+i\lambda-\nabla_\theta)^{-1}dz \\
= &\frac{1}{2\pi i} \int_{|z|=\delta}z(z+i\lambda-\nabla_\theta)^{-1} F_A((\frac{\p}{\p x^j})^h,\frac{\p}{\p \theta})(z+i\lambda-\nabla_\theta)^{-1}dz \\
= &\frac{1}{2\pi i} \int_{|z|=\delta}z(z+i\lambda-i\lambda_b-i\mu-im)^{-1} e^{i(\mu +m)\theta}\times \\
&  \times(F_{j\theta}^\mu )_a^bs_b(z+i\lambda -i\lambda_a-im)^{-1}dz \langle ,e^{im\theta}s_a\rangle_{L^2(S^1)} \\
=&  \left[\sum_{|\lambda-\lambda_b-\mu-m|<\delta}\frac{ \lambda_b+\mu+m -\lambda}{\lambda_b-\lambda_a +\mu}  -\sum_{|\lambda-\lambda_a-m|<\delta}\frac{\lambda_a+m-\lambda}{\lambda_b-\lambda_a +\mu}\right] e^{i(\mu +m)\theta}\times\\
&\hfill \times(F_{j\theta}^\mu )_a^bs_b\langle ,e^{im\theta}s_a\rangle_{L^2(S^1)} \\
= &   P^I_\lambda F_{j\theta}P^I_\lambda 
+    \sum_{\substack{|\lambda-\lambda_b-\mu-m|<\delta\\|\lambda-\lambda_a-m|>\delta}}\frac{ \lambda_b+\mu+m -\lambda}{\lambda_b-\lambda_a +\mu}    e^{i(\mu +m)\theta}(F_{j\theta}^\mu )_a^bs_b\langle ,e^{im\theta}s_a\rangle_{L^2(S^1)} \\
&-    \sum_{\substack{|\lambda-\lambda_b-\mu-m|>\delta\\|\lambda-\lambda_a-m|<\delta}}\frac{\lambda_a+m-\lambda}{\lambda_b-\lambda_a +\mu}  e^{i(\mu +m)\theta}(F_{j\theta}^\mu )_a^bs_b\langle ,e^{im\theta}s_a\rangle_{L^2(S^1)} \\
= &  P^I_\lambda F_{j\theta}P^I_\lambda
+O\left(|P^I_\lambda F_{j\theta}(I-P^I_\lambda)| + |(I-P^I_\lambda) F_{j\theta} P^I_\lambda |\right).
\end{align*}

 \begin{align}\label{deriv1}
\frac{\p}{\p x^j}\frac{1}{2}|\Phi^I_\lambda|^2& = -\frac{\p}{\p x^j}\Tr\frac{1}{2\pi i}\int_{|z|=\delta}\frac{z^2}{2}(z+i\lambda-\nabla_\theta)^{-1}dz %\nonumber\\
 = \langle\Phi^I_\lambda,F_{j\theta}\rangle.
 \end{align}

\begin{align}\label{amorpho}
\Delta\frac{1}{2}\Tr\, (\Phi^I_\lambda)^2 &= - \Tr\frac{1}{2\pi i}\int_{|z|=\delta}z(z+i\lambda-\nabla_\theta)^{-1}F_{j\theta}(z+i\lambda-\nabla_\theta)^{-1}F_{j\theta}dz\nonumber\\
&\quad -  \Tr\frac{1}{2\pi i}\int_{|z|=\delta}z(z+i\lambda-\nabla_\theta)^{-1}[\nabla_{(\frac{\p}{\p x^j})^h},F_{j\theta}]dz\nonumber\\
&=  \sum_{\substack{|\lambda- m- \lambda_a|<\delta\\|\lambda- m-\mu- \lambda_a|<\delta}} 
|(F_{j\theta}^\mu)_a^b|^2  \nonumber\\
&\quad - \sum_{\substack{|\lambda- m- \lambda_a|<\delta\\|\lambda- m-\mu- \lambda_a|>\delta}}  ( \lambda-m-\lambda_a ) (\lambda_a  -\mu-\lambda_b)^{-1}
|(F_{j\theta}^\mu)_a^b|^2   \nonumber\\
&\quad - \sum_{\substack{|\lambda- m- \lambda_a|>\delta\\|\lambda- m-\mu- \lambda_a|<\delta}}  (\lambda-m-\mu-\lambda_b)
(\mu+\lambda_b  -\lambda_a)^{-1}|(F_{j\theta}^\mu)_a^b|^2 dz \nonumber\\
&\quad -  \Tr\Phi^I_\lambda[\nabla_{(\frac{\p}{\p x^j})^h},F_{j\theta}].
\end{align}
We examine the last term. We compute 
\begin{align}
 &\sum_j \left(\nabla_{(\frac{\p}{\p x^j})^h}(\frac{\p}{\p x^j})^h\right)\wedge \frac{\p}{\p \theta}+ (\frac{\p}{\p x^j})^h\wedge \nabla_{ (\frac{\p}{\p x^j})^h}\frac{\p}{\p\theta} \nonumber\\
& = -\sum_j V^{-1}V_{,j}(\frac{\p}{\p x^j})^h \wedge \frac{\p}{\p \theta}+ (\frac{\p}{\p x^j})^h\wedge \frac{1}{2}  V^{-2}  (\omega_{i,j}-\omega_{j,i}) (\frac{\p}{\p x^i})^h \nonumber\\
& = \frac{k}{2lr^2}[(\frac{\p}{\p r})^h \wedge \frac{\p}{\p \theta}- \frac{1}{2l}  \ast dr(\frac{\p}{\p x^i},\frac{\p}{\p x^j})(\frac{\p}{\p x^i})^h\wedge (\frac{\p}{\p x^j})^h ]+O(r^{-3}).
\end{align}
Up to the $O(r^{-3})$ error term, this is metrically dual to the anti-self-dual  two-form  
\begin{align}
 &  -V^{-1}dV \wedge d\theta + \ast_{\IR^3}dV.
\end{align}
Hence  the value of \(F_A\) on this sum of two bivector is twice its value on the first bivector, so 
\begin{align}\label{verde}
[\nabla_{(\frac{\p}{\p x^j})^h},F_{j\theta}] = \frac{k}{lr^2}F_A((\frac{\p}{\p r})^h ,\frac{\p}{\p \theta})+O(r^{-5}),
\end{align}
and for any section $\phi$ of $ad(E)$, we have 
\begin{align}\label{topcharge}
\int_{S_R(0)}\Tr\, \phi [\nabla_{(\frac{\p}{\p x^j})^h},F_{j\theta}]d\sigma = \frac{k}{R^2}\int_{S_R(0)}\Tr\, \phi  F_{A}+\int_{S_R(0)}O(R^{-3}|\phi||F_A|)d\sigma  ,
\end{align}
where we have used the fact that $ -V^{-1}dV \wedge d\theta + \ast_{\IR^3}dV $ is the projection of $2\ast_{\IR^3}dV$ onto the anti-self-dual forms. With these observations, \eqref{amorpho} yields the following estimate. 
\begin{align}\label{bracket}
&\frac{1}{2}| F^\lambda |^2+O(r^{-2}|\Phi^I_\lambda||F_A|) -|F^B|^2 \leq \Delta\frac{1}{2}\Tr\, (\Phi^I_\lambda)^2\leq  \frac{1}{2}| F^\lambda |^2+O(r^{-2}|\Phi^I_\lambda||F_A|)+ |F^B|^2.
\end{align}
The $O(r^{-2}|\Phi^I_\lambda||F_A|)$ term is the only term in this expression that does not have a comparable term in the monopole analysis, but it is decaying more rapidly  than the expected $O(r^{-4})$ for the $O(|F_A|^2)$ terms. Hence, it does not materially alter the analysis carried out for monopoles.  Following the monopole argument in Section 4, and keeping track of the additional factors of $\frac{1}{2}$ in the instanton case, equation \eqref{sunwtintl} becomes 
\begin{align}\label{refuel}
 &\frac{1}{2}T^D_I(R)-T^B_I(R)-C\int_{B_R(0)^c}r^{-2}|\Phi^I_D||F_D|dv\nonumber\\
&\leq -\frac{1}{4}R\frac{\p}{\p R}T^D_I(R)+\frac{1}{4}T^D_I(R)-\frac{1}{4}\int_{B_R(0)^c}\frac{R}{r}|F^D|^2dv\nonumber\\
&\quad +\frac{1}{2}T^B_I(R)+ C\int_{B_R(0)^c}r^{-2}|\Phi^I_D||F_D|dv.
\end{align}
 Hence 
\begin{align}
 &  \frac{\p}{\p R}(RT^D_I(R)) \leq   - \int_{B_R(0)^c}\frac{R}{r}|F^D|^2dv+6T^B_I(R)+ 8C\int_{B_R(0)^c}r^{-2}|\Phi^I_D||F_D|dv,
\end{align}
and for all $\epsilon$, 
\begin{align}\label{srefuel}
  \frac{\p}{\p R}(R^{1-\epsilon}T^D_I(R)) \leq   &- R^{-\epsilon}\int_{B_R(0)^c}(\epsilon+\frac{R}{r}|F^D|^2)dv+6R^{-\epsilon}T^B_I(R)\nonumber\\
&+ 8CR^{-\epsilon}\int_{B_R(0)^c}r^{-2}|\Phi^I_D||F_D|dv,
\end{align}
Hence, either $R^{1-\epsilon}T^D_I(R)$ is decreasing or  
  \begin{align}\label{ineq:XieR}
 %\hspace*{-2.5cm}\!\!\!
 & \text{\bf Inequality  }  X_I(\epsilon,R) :& %\hspace{1.5cm} 
    &\int_{B_R(0)^c}(\epsilon+\frac{R}{r}|F^D|^2)dv %\nonumber\\
 \leq 6 T^B_I(R)+ 8C \int_{B_R(0)^c}r^{-2}|\Phi^I_D||F_D|dv. \nonumber
\end{align}

We now extend Proposition \ref{proppart} to instantons. 
\begin{proposition}\label{proppartI}
For every $\epsilon>0$,  there is $A_\epsilon>0$ such that 
\begin{align}
T_I(R)\leq A_\epsilon R^{\epsilon - 1}.
\end{align}
\end{proposition}
\begin{proof}
From \eqref{offdiagI} we have the estimate, for $R$ sufficiently large,
\begin{align}\label{offdiag3I}
 T^B_I(R)&\leq   4\delta^{-2}C_n \frac{T_I(\frac{R}{2})}{ R^2}+\delta^{-2}C_nT_I(\frac{R}{2})^{3-2(\frac{3}{4})^{n}}.
\end{align}  
 
If for some $\epsilon<1$, inequality  $X_I(\epsilon,R)$ holds, we have  
\begin{align}\label{crudeXI}
 T_I(R)& \leq   7\epsilon^{-1}\delta^{-2}C_n (4\frac{T_I(\frac{R}{2})}{ R^2}+ T_I(\frac{R}{2})^{3-2(\frac{3}{4})^{n}})+\epsilon^{-1}8C \int_{B_R(0)^c}r^{-2}|\Phi^I_D||F_D|dv\nonumber\\
& \leq   7\epsilon^{-1}\delta^{-2}C_n (4\frac{T_I(\frac{R}{2})}{ R^2}+ T_I(\frac{R}{2})^{3-2(\frac{3}{4})^{n}})+\epsilon^{-2}\|\Phi^I_D\|_{L^\infty(B_R(0)^c)} 32C^2 \int_{B_R(0)^c}r^{-4}dv\nonumber\\
&\quad+\frac{1}{2}
\|\Phi^I_D\|_{L^\infty(B_R(0)^c)} T_I(R).
\end{align}
For $R$ sufficiently large that $\|\Phi^I_D\|_{L^\infty(B_R(0)^c)} \leq \frac{1}{4}$, and for  $c:=\frac{2^{10}C^2\pi}{7}$, rewrite this inequality as 
\begin{align}\label{crudeX2I}
  T_I(R)& \leq    \frac{8C_n}{\epsilon\delta^{2}} (4\frac{T_I(\frac{R}{2})}{ R^2}+ T_I(\frac{R}{2})^{3-2(\frac{3}{4})^{n}})+\frac{c}{\epsilon^{2}}  \|\Phi^I_D\|_{L^\infty(B_R(0)^c)}R^{-1}.
\end{align}
Hence 
\begin{align}\label{crudeX3I}
 R^{1-\epsilon}T_I(R)& \leq    \frac{8C_n}{\epsilon\delta^{2}} 2^{1-\epsilon}(\frac{4}{ R^2}+ T_I(\frac{R}{2})^{2-2(\frac{3}{4})^{n}})(\frac{R}{2})^{1-\epsilon}T_I(\frac{R}{2})+\frac{c}{\epsilon^{2}}  \|\Phi^I_D\|_{L^\infty(B_R(0)^c)}R^{-\epsilon}.
\end{align}
There exists $R_{n,\epsilon}$ so that for $R\geq R_{n,\epsilon},$  one has   $\frac{8C_n}{\epsilon\delta^{2}} 2^{1-\epsilon}(\frac{4}{ R^2}+ T_I(\frac{R}{2})^{2-2(\frac{3}{4})^{n}})\leq \frac{1}{2}$, and $\frac{c}{\epsilon^{2}}  \|\Phi^I_D\|_{L^\infty(B_R(0)^c)}R^{-\epsilon}\leq \frac{1}{4}$; therefore, if inequality $X_I(\epsilon,R)$  holds for $R\geq R_{n,\epsilon}$, we have 
\begin{align}\label{step1I}
R^{1-\epsilon}T_I(R)& \leq   \max\{(\frac{R}{2})^{1-\epsilon}T_I(\frac{R}{2}), \frac{1}{4}\}.
\end{align} 
Given inequality \eqref{step1I}, the proof now proceeds exactly as the proof of Proposition \ref{proppart}, except now the iteration  can also stop (with a successful conclusion)  as soon as    
$\frac{1}{4}=\max\{(\frac{R}{2})^{1-\epsilon}T_I(\frac{R}{2}), \frac{1}{4}\}.$ 
We do not repeat the details. 
 
\end{proof}
With Proposition \ref{proppartI} in hand, the remaining estimates from Section 5 immediately extend to the instanton cases. We record them without further comment. 

\begin{theorem}\label{goodtheoremI}
For some $C_1,C_{2,\epsilon}>0$, $T_I(R)\leq C_1R^{-1},$ and $T^B_I(R)\leq C_{2,\epsilon}R^{-3+\epsilon}.$
\end{theorem}
\begin{corollary}
$S_I(R)=O(R^{-2}).$ 
\end{corollary}
\begin{corollary}
$|\Phi^I_\lambda|= O(R^{-1}).$
\end{corollary}
\begin{corollary}\label{grad1I}
$\|R^{2+m}\nabla^m F_A\|_{L^\infty}<\infty.$
\end{corollary}
\begin{corollary}\label{Best1I}
$$T^B_I(R) = O(\delta^{-2m}R^{-1-2m}), \forall m,$$
and
$$ \|R^{m} F_A^B\|_{L^\infty}<\infty, \forall m.$$
\end{corollary}
\section{Angular Instanton Estimates}
The results from Section \ref{angsec} extend with little change to the instanton context. 
\begin{proposition}\label{exper1I}
$$\tau_I(L):= \int_{B_L(0)^c} \frac{1}{r}|\nabla (r\Phi^I_{\lambda})|^2  dv<\infty,$$ 
for some $c>0$, 
$$|\nabla (r\Phi^I_{\lambda})|^2\leq cr^{-2}\tau_I(\frac{r}{2}).$$
and 
$$\hat\tau_I(L):=\int_{ B_L(0)^c} r|\nabla(\nabla_r(r\Phi^I_{\lambda}))|^2dv<\infty.$$
\end{proposition}
\begin{proof}
The proof is identical to the proof of Lemma \ref{exper1} once we check that the analogs of \eqref{boeing} and \eqref{autumn} hold in the instanton context. 
\begin{align}\label{boeingI}
\Delta \frac{1}{2}|r\Phi^I_\lambda|^2 &=   4r\Tr\, \Phi^I_\lambda  F_{r\theta}  
-   \frac{1}{2}r^2 |  F^{\lambda}|^2 +  r^2\Tr\, \Phi^I_\lambda[\nabla_{(\frac{\p}{\p x^j})^h},F_{j\theta}]
-3|\Phi^I_\lambda|^2 \nonumber\\
&\quad +  2\Tr\frac{1}{2\pi i}P^I_\lambda \int_{\mathcal{C}_\lambda}  zr^2(z+i\lambda-i\nabla_\theta)^{-1}F_{j\theta}
 (z+i\lambda-i\nabla_\theta)^{-1}(I-P^I_\lambda )F_{j\theta}dzP^I_\lambda\nonumber\\
&=  - 4r\langle \Phi^I_\lambda , F_{r\theta}  \rangle
-   \frac{1}{2}  | r F^{\lambda}|^2 -3|\Phi^I_\lambda|^2 +O(r^{-3})\nonumber\\
&= -  |   \nabla(r\Phi^I_\lambda)|^2
-    \frac{1}{r}\nabla_r |r\Phi^I_\lambda |^2 +O(r^{-3}).
\end{align}
This is the requisite analog of \eqref{boeing} to the instanton context. 
We estimate $\nabla^*\nabla\Phi^I_\lambda.$
\begin{align}
\nabla^*\nabla\Phi^I_\lambda&=-[\nabla_{(\frac{\p}{\p x^j})^h},[\nabla_{(\frac{\p}{\p x^j})^h},\Phi^I_\lambda]]\nonumber\\
&= \frac{2}{2\pi i} \int_{|z|=\delta}z (z+i\lambda-\nabla_\theta)^{-1} F_A((\frac{\p}{\p x^j})^h,\frac{\p}{\p \theta})\times\nonumber\\
&\phantom{F_A((\frac{\p}{\p x^j})^h,\frac{\p}{\p \theta})}
\times(z+i\lambda-\nabla_\theta)^{-1} F_A((\frac{\p}{\p x^j})^h,\frac{\p}{\p \theta})(z+i\lambda-\nabla_\theta)^{-1}dz\nonumber\\
&\quad + \frac{1}{2\pi i} \int_{|z|=\delta}z(z+i\lambda-\nabla_\theta)^{-1} [\nabla_{(\frac{\p}{\p x^j})^h},F_A((\frac{\p}{\p x^j})^h,\frac{\p}{\p \theta})](z+i\lambda-\nabla_\theta)^{-1}dz\nonumber\\
&= P^I_\lambda [\nabla_{(\frac{\p}{\p x^j})^h},F_A((\frac{\p}{\p x^j})^h,\frac{\p}{\p \theta})] +O(\text{commutators}) \nonumber\\
&= P^I_\lambda [\nabla_{(\frac{\p}{\p x^j})^h},F_A((\frac{\p}{\p x^j})^h,\frac{\p}{\p \theta})] +O(r^{-5})
\nonumber\\
&= \frac{k}{lr^2}\nabla_r\Phi^I_\lambda +O(r^{-5}).
\end{align}
In the preceding calculation, we have simplified the integrals by  shifting the curvature terms all the way to the right, acquiring commutators in the process. These resulting commutators of the curvature terms with $(z+i\lambda-\nabla_\theta)^{-1}$ are given schematically by  
$$[\mathcal{F},(z+i\lambda-\nabla_\theta)^{-1}]= -(z+i\lambda-\nabla_\theta)^{-1}[\nabla_\theta,\mathcal{F}](z+i\lambda-\nabla_\theta)^{-1}.$$
The $O(r^{-5})$ now follows from Corollary \ref{grad1I} (and computations of $\nabla_\theta (\frac{\p}{\p x^j})^h$ and $\nabla_\theta \frac{\p}{\p \theta}$). 
 We also have 
\begin{align}\label{3xcheck}
\Delta \frac{1}{2}|\nabla (r\Phi^I_\lambda)|^2
&= -|\nabla\nabla (r\Phi^I_\lambda)|^2+\langle \nabla^*\nabla \nabla (r\Phi^I_\lambda),\nabla (r\Phi^I_\lambda)\rangle \nonumber\\
&= -|\nabla\nabla (r\Phi^I_\lambda)|^2-\langle \nabla_m\nabla_j \nabla_j (r\Phi^I_\lambda),\nabla_m (r\Phi^I_\lambda)\rangle 
-\langle \nabla_j[F_{jm},r\Phi^I_\lambda ],\nabla_m (r\Phi^I_\lambda)\rangle \nonumber\\
&\quad
-\langle [F_{ jm}-d\omega_{jm}\nabla_\theta,\nabla_j  r\Phi^I_\lambda],\nabla_m (r\Phi^I_\lambda)\rangle \nonumber\\
&= -|\nabla\nabla (r\Phi^I_\lambda)|^2-\langle \nabla_m\nabla_j \nabla_j (r\Phi^I_\lambda),\nabla_m (r\Phi^I_\lambda)\rangle 
-\langle \nabla_j[F_{jm},r\Phi^I_\lambda ],\nabla_m (r\Phi^I_\lambda)\rangle \nonumber\\
&\quad
-\langle [F_{ jm}-d\omega_{jm}\nabla_\theta,\nabla_j  r\Phi^I_\lambda],\nabla_m (r\Phi^I_\lambda)\rangle \nonumber\\
&= -|\nabla\nabla (r\Phi^I_\lambda)|^2 +\langle \nabla_mr\nabla^*\nabla  \Phi^I_\lambda ,\nabla_m (r\Phi^I_\lambda)\rangle 
-2\langle  \nabla_r\nabla_m   \Phi^I_\lambda ,\nabla_m (r\Phi^I_\lambda)\rangle 
\nonumber\\
&\quad -2\langle  [F_{mr}, \Phi^I_\lambda] ,\nabla_m (r\Phi^I_\lambda)\rangle 
-2\langle  \frac{1}{r}\nabla_m   \Phi^I_\lambda ,\nabla_m (r\Phi^I_\lambda)\rangle 
+2\langle  \frac{1}{r}\nabla_r \Phi^I_\lambda ,\nabla_r (r\Phi^I_\lambda)\rangle 
\nonumber\\
&\quad
-\langle \nabla_m  \frac{2}{r}(\Phi^I_\lambda),\nabla_m (r\Phi^I_\lambda)\rangle 
-\langle \nabla_j[F_{jm},r\Phi^I_\lambda ],\nabla_m (r\Phi^I_\lambda)\rangle \nonumber\\
&\quad
-\langle [F_{ jm} ,\nabla_j  (r\Phi^I_\lambda)],\nabla_m (r\Phi^I_\lambda)\rangle 
+\langle d\omega_{jm}[  F_{\theta j} , r\Phi^I_\lambda]],\nabla_m (r\Phi^I_\lambda)\rangle\nonumber\\
&= -|\nabla\nabla (r\Phi^I_\lambda)|^2  
- r^{-1} \nabla_r|\nabla ( r\Phi^I_\lambda )|^2
+2r^{-2}| \nabla_r(r\Phi^I_\lambda )|^2 -\frac{2}{r^2}|\nabla^0  (r\Phi^I_\lambda)|^2
 \nonumber\\
&\quad -2\langle  [F_{mr}, \Phi^I_\lambda] ,\nabla_m (r\Phi^I_\lambda)\rangle 
-2\langle [F_{ jm} ,\nabla_j  (r\Phi^I_\lambda)],\nabla_m (r\Phi^I_\lambda)\rangle +  O(r^{-5}) \nonumber\\
&\leq \frac{c}{r^2}| \nabla (r\Phi^I_\lambda)|^2   +  O(r^{-5}) ,
\end{align}
for some $c>0$. 
This is the desired instanton analog of \eqref{autumn}. The remainder of the proof is the same as for Proposition \ref{exper1}.
\end{proof}
We also have the instanton version of Corollary \ref{shortcut}
\begin{corollary}\label{Ishortcut}
\begin{align}|F_A(\frac{\p}{\p r},\cdot)|= O(r^{-2} \tau_I^{\frac{1}{2}}(\frac{r}{2})),
\end{align}
\begin{align}\label{Itreefrog}
\Phi_\lambda^I = -r\ast F_A^\lambda(\frac{\p}{\p r}) + O(r^{-1}\tau_I^{\frac{1}{2}}(\frac{r}{2})),
\end{align}
and 
\begin{align}
\nabla^0\Phi_\lambda^I= O(r^{-2}\tau_I^{\frac{1}{2}}(\frac{r}{2})).
\end{align}
\end{corollary}
\section{Instanton Connection Asymptotics}
Let $i\lambda\in \spec (\nabla_\theta(\infty)).$ Let $d_\lambda:=  \rank P^I_\lambda$. Given $\alpha >0$ and $p\in \IR^3$, $|p|$ large, decompose $[1,\ldots,d_\lambda]$ into disjoint subsets $J_i^\alpha(p,\lambda)$ as described in Section 7. 
Let $P^I_{J_i^\alpha(p,\lambda)}$ and $\Phi^I_{J_i^\alpha(p,\lambda)}$ denote respectively the associated projections and the summands of $\Phi^I_\lambda$, and let 
$F^{J_i^\alpha(p,\lambda)}:= P^I_{J_i^\alpha(p,\lambda)}F_AP^I_{J_i^\alpha(p,\lambda)} $. 
To extend the monopole connection asymptotics to the instanton case, we apply Chern-Weil theory to $E_{J_i^\alpha(p,\lambda)}$. We follow the proof of \cite[Theorem 23]{First}. 

Denote the eigenvalues of $\Phi^I_{\lambda}(x)$ by $\{\frac{a_j(x,\lambda)}{r}\}_{j=1}^{d_\lambda}$. 
In a neighborhood of $p$, we define $W(J_i^\alpha(p,\lambda))$ to be the hermitian vector bundle with fiber 
$$W(J_i^\alpha(p,\lambda))_x:= \mathop{\oplus}_{j \in E_{J_i^\alpha(p,\lambda)}}\Ker(\nabla_{\theta}-i\lambda-i\frac{a_j(x,\lambda)}{r} )\subset L^2(\Pi^{-1}(x),E),$$
and hermitian structure on the fiber induced from its inclusion in $L^2(\Pi^{-1}(x),E).$
This bundle is well defined on the closure of a connected neighborhood $U$ of $p$ if and only if $\exists \beta>0$ such that $|a_m(x,\lambda)-a_{l}(x,\lambda)|>\beta$, for all $x\in U$, and for all $(m,l)$ with $m\in J_i^\alpha(p,\lambda)$, and $l\not \in J_i^\alpha(p,\lambda).$ %We call such a $\beta$ a  $J_i^\alpha(p,\lambda)$ spectral gap on $U$.
 Moreover, $\forall\beta>0$, $\exists L_\beta>0$ so that for $|x|>L_\beta$, if  $|a_m(x,\lambda)-a_{l}(x,\lambda)|>\beta$, for all $(m,l)$ with $m\in J_i^\alpha(p,\lambda)$ then $W(J_i^\alpha(p,\lambda))$ is a well defined vector bundle on 
  the sphere $S_{|x|}(0)$.  
 The bundle $W(J_i^\alpha(p,\lambda))$ inherits the connection 
$$\nabla_{\frac{\p}{\p x^j}}^{P_{J_i^\alpha(p,\lambda)}} := P^I_{J_i^\alpha(p,\lambda)} \nabla_{(\frac{\p}{\p x^j})^h}P^I_{J_i^\alpha(p,\lambda)},$$
with curvature 
\begin{align}
F_{J_i^\alpha(p,\lambda)}(\frac{\p}{\p x^j},\frac{\p}{\p x^m}) &=  F^{J_i^\alpha(p,\lambda)}((\frac{\p}{\p x^j})^h,(\frac{\p}{\p x^m})^h)- d\omega(\frac{\p}{\p x^j},\frac{\p}{\p x^m})P^I_{J_i^\alpha(p,\lambda)} \nabla_{\frac{\p}{\p \theta}}P^I_{J_i^\alpha(p,\lambda)}\nonumber\\
&\quad +P^I_{J_i^\alpha(p,\lambda)} [[\nabla_{(\frac{\p}{\p x^j})^h},P^I_{J_i^\alpha(p,\lambda)}], [\nabla_{(\frac{\p}{\p x^m})^h},P^I_{J_i^\alpha(p,\lambda)} ]]P^I_{J_i^\alpha(p,\lambda)}.
\end{align}
We have 
\begin{align}
- d\omega(\frac{\p}{\p x^j},\frac{\p}{\p x^m})P^I_{J_i^\alpha(p,\lambda)} \nabla_{\frac{\p}{\p \theta}}P^I_{J_i^\alpha(p,\lambda)}
= \frac{i\lambda k}{2r^2} \ast dr(\frac{\p}{\p x^j},\frac{\p}{\p x^m}) P^I_{J_i^\alpha(p,\lambda)}+O(r^{-3}),
\end{align}
and
\begin{align}
 [\nabla_{(\frac{\p}{\p x^j})^h}&,P^I_{J_i^\alpha(p,\lambda)}(x)]=\\
 &=\frac{1}{2\pi i}\int_{\mathcal{C}_{J_i^\alpha(p,\lambda)}} (z+i\lambda-\nabla_\theta)^{-1}F_{j\theta}(z+i\lambda-\nabla_\theta)^{-1}dz \nonumber\\
&=\frac{1}{2\pi i}\int_{\mathcal{C}_{J_i^\alpha(p,\lambda)}} (z+i\lambda-\nabla_\theta)^{-1}P^I_{J_i^\alpha(p,\lambda)}F_{j\theta}(I-P^I_{J_i^\alpha(p,\lambda)})(z+i\lambda-\nabla_\theta)^{-1}dz \nonumber\\
&\quad+\frac{1}{2\pi i}\int_{\mathcal{C}_{J_i^\alpha(p,\lambda)}} (z+i\lambda-\nabla_\theta)^{-1}(I-P^I_{J_i^\alpha(p,\lambda)})F_{j\theta}P^I_{J_i^\alpha(p,\lambda)}(z+i\lambda-\nabla_\theta)^{-1}dz\nonumber\\
& = O(\frac{1}{\beta |x|^2}).
\end{align}
Hence 
\begin{align}
F_{J_i^\alpha(p,\lambda)}(\frac{\p}{\p x^j},\frac{\p}{\p x^m})(x) =&  \frac{i\lambda k}{2r^2} \ast dr(\frac{\p}{\p x^j},\frac{\p}{\p x^m}) P^I_{J_i^\alpha(p,\lambda)}\nonumber\\
&+F^{J_i^\alpha(p,\lambda)}((\frac{\p}{\p x^j})^h,(\frac{\p}{\p x^m})^h)
+O(\frac{1}{\beta^2 |x|^4}).
\end{align}
We have 
\begin{align}
\frac{\p}{\p x^j}\Tr\Phi^I_{J_i^\alpha(p,\lambda)}=  \Tr\, F^{J_i^\alpha(p,\lambda)}_{j\theta},
\end{align}
and
\begin{align}
\Delta\Tr\Phi^I_{J_i^\alpha(p,\lambda)} =&    ( i\lambda_a-i\lambda_b-i\mu)^{-1} \frac{1}{2}|((I-P_{J_i^\alpha(p,\lambda)})F_{A}^\mu P_{J_i^\alpha(p,\lambda)})_a^b|^2 \nonumber\\
&
+ ( i\lambda_b +i\mu  -i\lambda_a )^{-1}  \frac{1}{2}|(P_{J_i^\alpha(p,\lambda)}F_{A}^\mu(I-P_{J_i^\alpha(p,\lambda)}))_a^b|^2 \nonumber\\
&
-  \Tr\, P_{J_i^\alpha(p,\lambda)} [\nabla_{(\frac{\p}{\p x^j})^h},F_{j\theta}]
.\end{align}
We observe the following immediate consequence of \eqref{verde} and Corollary \ref{grad1I}. 
\begin{align}\label{verde2}
P^I_\lambda[\nabla_{(\frac{\p}{\p x^j})^h},F_{j\theta}] = \frac{k}{lr^2}\nabla_r\Phi^I_\lambda+O(r^{-5}).
\end{align}

We now have the instanton analog  of Proposition \ref{lunch} and Theorem \ref{specf} combined. 
 
\begin{theorem}\label{specfI}
For each $i\lambda\in \spec (\nabla_\theta(\infty))$, the spectrum of $|p|\Phi^I_\lambda(p)$ has a well defined limit as $|p|\to \infty$. Hence the spectrum of $\nabla_\theta$ has the form $\{i \lambda_a+i \frac{\beta_j}{r}+o(\frac{1}{r})\}_{a,j}$.  At large radius, $E$ decomposes into a direct sum of vector bundles $E=\oplus E(\lambda,\beta)$ such that $\Phi^I$ acts as $i \lambda+ i \frac{\beta}{r}+o(r)$ on $E(\lambda,\beta)$.  Moreover, on $\IR^3\setminus B_R(0)$, for $R$ large,  there exist bundles with connection $(W(\lambda,\beta),d_{W(\lambda,\beta)})$ such  that 
$$ E(\lambda,\beta) = \Pi^*W(\lambda,\beta),$$
and 
\begin{align}\label{form2}d_A = \mathop{\oplus}_{\lambda,\beta} \left(\Pi^*d_{W(\lambda,\beta)}+ i\lambda+\frac{i\beta}{r}+o(\frac{1}{r})\right).
\end{align}
Moreover,
\begin{align}\label{integrality}(l\beta-k\lambda)\in  \frac{1}{2}\IZ.
\end{align}
Let $P_{\lambda,\beta}$ denote unitary projection onto $E(\lambda ,\beta)$. Then 
for $\lambda_1\not = \lambda_2$, 
$$ |P_{\lambda_1,\beta_1}d_A P_{\lambda_2,\beta_2}|= O(r^{-N}), \forall N,$$ 
and for $\beta_1\not = \beta_2$ 
$$|P_{\lambda,\beta_1}d_A P_{\lambda,\beta_2}|= o(r^{-1}).$$ 
\end{theorem}
\begin{proof}
The proof of the structure of the limiting spectrum is the same as for Theorem \ref{specf}. We are left to prove the statement about subbundles. 
Let $i\lambda\in \spec (\nabla_\theta(\infty))$. Let $\beta_1(\lambda)<\beta_2(\lambda)<\ldots<\beta_{l(\lambda)}(\lambda)$ denote the distinct elements  of the limiting spectrum of $|p|\Phi^I_\lambda(p)$. Let $t_j(\lambda)$ denote the multiplicity of $\beta_j(\lambda)$. 
Then we have $\{1,\ldots,d_\lambda\} = J_1(\lambda)\cup J_2(\lambda)\cdots\cup J_{l(\lambda)}(\lambda),$ 
$J_1(\lambda) := \{1,\ldots,t_1(\lambda)\},$ $J_2:= \{t_1+1,\ldots,t_1(\lambda)+t_2(\lambda)\},$ and in general, 
$J_i(\lambda):= \{ 1+\sum_{j< i}t_j(\lambda),\ldots,\sum_{j\leq i}t_j(\lambda) \}.$ This decomposition is of the form $J_i(\lambda): = J_i^\alpha(p,\lambda)$ for some $\alpha$ sufficiently small and all $p$ sufficiently large. Then on the complement of some compact set, the vector bundles $W(J_i(\lambda)):= W(J_i^\alpha(p,\lambda))$ are well defined rank $t_i(\lambda)$ bundles. Setting 
$$E(\lambda,\beta_j(\lambda):= \Pi^*W(J_i(\lambda)),$$
we have 
$$E=\mathop{\oplus}_{\lambda,j}E(\lambda,\beta_j(\lambda)),$$
as claimed. 

In the instanton context \eqref{crapaud} becomes 
\begin{align}
\Phi^I_{J^\alpha_k(p,\lambda)}&= -rF_A^{J^\alpha_k(p,\lambda)}((\frac{\p}{\p r})^h,\frac{\p}{\p \theta}) 
+ O(r^{-1}\tau_I^{\frac{1}{2}}(\frac{r}{2}))\nonumber\\
&= -rV^{-1}F_A^{J^\alpha_k(p,\lambda)}(e_1^h,e_2^h) + O(r^{-1}\tau_I^{\frac{1}{2}}(\frac{r}{2}))\nonumber\\
&= - \frac{r}{l}F_{J^\alpha_k(p,\lambda)}(e_1^h,e_2^h)+\frac{i\lambda k}{2lr}P_{J^\alpha_k(p,\lambda)} + O(r^{-1}\tau_I^{\frac{1}{2}}(\frac{r}{2})),
\end{align}
where $\{e_1,e_2\}$ is a local oriented orthonormal frame for $\langle \frac{\p}{\p r}\rangle^{\perp}$. Hence, \eqref{vix} becomes 
\begin{align}
\frac{l}{ \rank (E_{J^\alpha_k(p,\lambda)})}\tr  \Phi^I_{J^\alpha_k(p,\lambda)}(q)= \frac{i}{2r}\frac{c_1(E_{J^\alpha_k(p,\lambda)})}{ \rank (E_{J^\alpha_k(p,\lambda)})}+\frac{i\lambda k}{2r}+o(r^{-1}).
\end{align}
As observed in the proof of Theorem \ref{specf}, polystability implies $\frac{c_1(E_{J^\alpha_k(p,\lambda)})}{ \rank (E_{J^\alpha_k(p,\lambda)})}\in \IZ$, 
which implies $(l\beta-\frac{k\lambda}{2})\in  \frac{1}{2}\IZ.$

Let $d_W$ denote the connection on $\oplus_{\lambda,j}W_\lambda(J_j)$ obtained from the sum of the connections defined on each summand. 
We write
\begin{align}
\nabla_{(\frac{\p}{\p x^j})^h}- \nabla_{(\frac{\p}{\p x^j})^h}^{\Pi^*d_W}  
&= \nabla_{(\frac{\p}{\p x^j})^h} - \sum_{\lambda,i}P_{J_i(\lambda)}\nabla_{(\frac{\p}{\p x^j})^h}P_{J_i(\lambda)}\nonumber\\
&= [(I-P_{J_i(\lambda)}),[\nabla_{(\frac{\p}{\p x^j})^h},P_{J_i(\lambda)}]].  
\end{align}
Choose a frame so that if $P_{J_i(\lambda)}e^{im\theta}s_c\not = 0$, then $m=0$, and for such an $s_c$,  
we have
\begin{align}\label{wren}
&[\nabla_{(\frac{\p}{\p x^j})^h},P_{J_i(\lambda)}]P_{J_i(\lambda)} s_c\nonumber\\
& = \frac{1}{2\pi i}\int_{\mathcal{C}_{J_i(\lambda)}}(z+i\lambda-\nabla_\theta)^{-1}e^{i\mu\theta}(F_{j\theta}^\mu)_a^bs_b\langle\cdot ,s_a\rangle(z+i\lambda-\nabla_\theta)^{-1}dz P_{J_i(\lambda)}s_c\nonumber\\
& =  (i\lambda_c -i\mu-i\lambda_b)^{-1}e^{i \mu\theta}(F_{j\theta}^\mu)_c^bs_b .
\end{align} 
This term is $O(r^{-2})$ unless $\mu=0$ and $|\lambda_b-\lambda_c|= O(r^{-1}).$ In the latter case, it is $o(r^{-1})$. To see this, note that by Proposition \ref{exper1I} $|P^I_\lambda F_A(\frac{\p}{\p x^j})^h,\frac{\p}{\p r})^hP^I_\lambda -\frac{1}{r}\Phi^I_\lambda|= o(\frac{1}{r^2})$, and therefore $|P_{J_i(\lambda)} F_A(\frac{\p}{\p x^j})^h,\frac{\p}{\p r})^h(1-P_{J_i(\lambda)})|= o(\frac{1}{r^2}),$ and \eqref{form2} follows.  
\end{proof}

%\bibliographystyle{amsalpha}

%\providecommand{\bysame}{\leavevmode\hbox to3em{\hrulefill}\thinspace}
%\providecommand{\MR}{\relax\ifhmode\unskip\space\fi MR }
%\providecommand{\MRhref}[2]{%
%  \href{http://www.ams.org/mathscinet-getitem?mr=#1}{#2}
%}
%\providecommand{\href}[2]{#2}
\bibliographystyle{alphaurl}

 \end{document}